\documentclass[reqno,12pt]{amsart}
\usepackage{amscd}
\usepackage{amssymb}
\usepackage{verbatim}
\usepackage{amsmath}
\usepackage{amsxtra}
\usepackage{cite}           

\theoremstyle{plain}
\newtheorem{Theorem}{Theorem}[section]
  
\newtheorem{Lemma}[Theorem]{Lemma}
  
\newtheorem{Corollary}[Theorem]{Corollary}
  
\newtheorem{Proposition}[Theorem]{Proposition}

\theoremstyle{definition}
\newtheorem{Definition}[Theorem]{Definition}

\newtheorem{Remark}[Theorem]{Remark}

\numberwithin{equation}{section}
  
\newcommand{\const}{\operatorname{\text{\rm const}}}

\newcommand{\Nbb}{{\mathbb N}}

\newcommand{\Rbb}{{\mathbb R}}

\newcommand{\Ubb}{{\mathbb U}}

\newcommand{\Zbb}{{\mathbb Z}}

\newcommand{\F}{{\mathcal F}}

\renewcommand{\H}{{\mathcal H}}
\newcommand{\I}{{\mathcal I}}
\newcommand{\J}{{\mathcal J}}

\newcommand{\M}{{\mathcal M}}
\newcommand{\N}{{\mathcal N}}
\renewcommand{\O}{{\mathcal O}}
\renewcommand{\P}{{\mathcal P}}

\newcommand{\tR}{{\tilde R}}

\newcommand{\tX}{{\tilde X}}

\newcommand{\hX}{{\hat X}}

\newcommand{\hU}{{\hat U}}

\newcommand{\tU}{{\tilde U}}

\newcommand{\Ystar}{Y_{\nhspt *}}
\newcommand{\Xstar}{X_*}
\newcommand{\Ustar}{U_{\nhspt *}}
\newcommand{\hYstar}{\hat Y_{\nhspt *}}
\newcommand{\hXstar}{\hat X_*}
\newcommand{\hUstar}{\hat U_{\nhspt *}}

\newcommand{\Ya}{Y_{\nhspt a}}
\newcommand{\Yb}{Y_{b}}

\newcommand{\tY}{\tilde Y}

\newcommand{\bY}{\bar Y}

\newcommand{\Xa}{X_a}
\newcommand{\Ua}{U_a}

\newcommand{\Xb}{X_b}
\newcommand{\Ub}{U_b}

\newcommand{\hZ}{\hat Z}

\newcommand{\hY}{\hat Y}

\newcommand{\diag}{\operatorname{\text{\rm diag}}}
\newcommand{\Ker}{\operatorname{\text{\rm Ker}}}
\renewcommand{\Im}{\operatorname{\text{\rm Im}}}
\newcommand{\rank}{\operatorname{\text{\rm rank}}}

\newcommand{\defect}{\operatorname{\text{\rm def}\hspace*{1pt}}}
\newcommand{\ind}{\operatorname{\text{\rm ind}}}

\newcommand{\Arg}{\operatorname{\text{\rm Arg}}}
\newcommand{\ArgI}{\operatorname{\text{\rm Arg}}_{\hspt\I}}

\newcommand{\ArgL}{\operatorname{\text{\rm Arg}}_{\hspt[0,2\pi)}}
\newcommand{\ArgR}{\operatorname{\text{\rm Arg}}_{\hspt(0,2\pi]}}
\newcommand{\Mas}{\operatorname{\text{\rm Mas}}\hspt}

\newcommand{\mmatrix}[1]{\left(\begin{matrix} #1
  \end{matrix}\right)}

\newcommand{\qtextq}[1]{ \quad\text{ #1 }\quad }

\newcommand{\eps}{\varepsilon}
\newcommand{\la}{\lambda}

\newcommand{\hspt}{\hspace*{1pt}}
\newcommand{\nhspt}{\hspace*{-1pt}}

\newcommand{\Nab}[1]{\N(#1,[a,b])}
\newcommand{\Masab}[2]{\Mas(#1,#2,[a,b])}
\newcommand{\Nsab}[1]{\N^*(#1,[a,b])}
\newcommand{\Massab}[2]{\Mas^{\!\nhspt*}(#1,#2,[a,b])}

\newcommand{\GGamma}{\Gamma}
\newcommand{\ggamma}{\gamma}

\begin{document}


\thispagestyle{empty}
\renewcommand{\thefootnote}{}

\title[Oscillation numbers for Lagrangian paths]{%
  Oscillation numbers for continuous Lagrangian paths \\[1mm] and Maslov index}

\author{Julia Elyseeva}
\address{Department of Applied Mathematics, Moscow State University of Technology,
  Vadkovskii per.~3a, 101472, Moscow, Russia, E-mail address:
  {\tt elyseeva@mtu-net.ru}}

\author{Peter \v{S}epitka}
\address{Department of Mathematics and Statistics, Faculty of Science, Masaryk University,
  Kot\-l\'a\v r\-sk\'a~2, CZ-61137 Brno, Czech Republic, E-mail address:
  {\tt sepitkap@math.muni.cz}}

\author{Roman \v{S}imon Hilscher}
\address{Department of Mathematics and Statistics, Faculty of Science, Masaryk University,
  Kot\-l\'a\v r\-sk\'a~2, CZ-61137 Brno, Czech Republic, E-mail address:
  {\tt hilscher@math.muni.cz}}

\thanks{This research was supported by the Czech Science Foundation under grant
  GA19--01246S}


\begin{abstract}
  In this paper we present the theory of oscillation numbers and dual oscillation numbers
  for continuous Lagrangian paths in $\mathbb{R}^{2n}$. Our main results include a~connection
  of the oscillation numbers of the given Lagrangian path with the Lidskii angles of a~special
  symplectic orthogonal matrix. We also present Sturmian type comparison and separation
  theorems for the difference of the oscillation numbers of two continuous Lagrangian paths.
  These results, as well as the definition of the oscillation number itself, are based
  on the comparative index theory (Elyseeva, 2009). The applications of these results
  are directed to the theory of Maslov index of two continuous Lagrangian paths. We derive
  a~formula for the Maslov index via the Lidskii angles of a~special symplectic orthogonal
  matrix, and hence we express the Maslov index as the oscillation number of a~certain
  transformed Lagrangian path. The results and methods are based on a~generalization of the
  recently introduced oscillation numbers and dual oscillation numbers for conjoined bases
  of linear Hamiltonian systems (Elyseeva, 2019 and 2020) and on the connection between the
  comparative index and Lidskii angles of symplectic matrices (\v{S}epitka and
  \v{S}imon Hilscher, 2020).
\end{abstract}

\footnotetext{
  {\it Secondary\/} 53D12.}

\thanks{{\it Version: \today. Submitted to Journal of Dynamics and Differential Equations on 18.02.2021.}}

\maketitle

\renewcommand{\thefootnote}{{\bf\,\alph{footnote}\alph{footnote}\alph{footnote}\,}}

\section{Introduction} \label{S:Intro}
Let $n\in\Nbb$ be a~given dimension and $[a,b]\subseteq\Rbb$ a~given interval. In this
paper we develop the theory of oscillation numbers for arbitrary continuous Lagrangian
paths on $[a,b]$. A~continuous matrix-valued function $Y:[a,b]\to\Rbb^{2n\times n}$ is
a~{\em Lagrangian path\/} if
\begin{equation} \label{E:Y.Lagrangian.path.def}
  Y^T\nhspt(t)\hspt\J\hspt Y(t)=0, \quad \rank Y(t)=n, \quad t\in[a,b],
\end{equation}
where $\J\in\Rbb^{2n\times2n}$ is the canonical skew-symmetric matrix. When
a~Lagrangian path is constant on $[a,b]$, then we call it a~{\em Lagrangian plane}.
Lagrangian paths arise, among others, as particular solutions (called conjoined or
isotropic bases) of the linear Hamiltonian differential system
\begin{equation} \label{H}
  y'=\J\hspt\H(t)\,y, \quad t\in[a,b], \tag{H}
\end{equation}
where the coefficient matrix $\H:[a,b]\to\Rbb^{2n\times2n}$ is symmetric and piecewise
continuous. In this case the function $Y$ is piecewise continuously differentiable on
$[a,b]$ and satisfies \eqref{E:Y.Lagrangian.path.def}. We partition the matrices $\J$,
$Y(t)$, and $\H(t)$ into $n\times n$ blocks as
\begin{equation} \label{E:JYH.partition}
  \J=\mmatrix{0 & I \\ -I & 0}, \quad Y(t)=\mmatrix{X(t) \\ U(t)}, \quad
  \H(t)=\mmatrix{-C(t) & A^T\nhspt(t) \\ A(t) & B(t)}.
\end{equation}
The motivation for the present study comes from the qualitative theory of canonical
systems of the form \eqref{H} in
\cite{oD2017a,jvE2016,jvE2020,jvE2020a,pS.rSH2017,pS.rSH2020?e,pS.rSH2020?g}.

The oscillation number $\Nab{Y}$ and the dual oscillation number $\Nsab{Y}$ of
a~conjoined basis $Y$ of \eqref{H} on $[a,b]$ were defined in \cite{jvE2020,jvE2020a}
and \cite{jvE2020?c} as quantities describing the oscillations of the first
component $X$ of the conjoined basis $Y$. They are based on the notions of the
comparative index and the dual comparative index, see \cite{jvE2007,jvE2009a} or
\cite[Section~3]{oD.jvE.rSH2019}, and on piecewise constant symplectic transformations of
solutions of system \eqref{H} from \cite{aaA2001,aaA2011}. According to the
definitions in \eqref{E:comparative.index.def} below, for two (constant) Lagrangian
planes $Y$ and $\hY$ the comparative index $\mu(Y,\hY)$ and the dual comparative
index $\mu^*(Y,\hY)$ are integers between $0$ and $n$, which are defined in algebraic
way from the rank and the index of certain $n\times n$ matrices constructed from the
blocks of the matrices $Y$ and $\hY$. The utility of the comparative index for the
oscillation and spectral theory of system \eqref{H}, as well as for the parallel
theory of symplectic difference systems, is documented  e.g. in
\cite{jvE2016,jvE2017,jvE2018,jvE2019,pS.rSH2017,pS.rSH2019,pS.rSH2020,pS.rSH2021} and
\cite{oD.jvE.rSH2019,jvE2020b}. If system \eqref{H} satisfies the Legendre condition
\begin{equation} \label{E:LC}
  B(t)\geq0 \quad\text{for all } t\in[a,b],
\end{equation}
then the oscillation number $\Nab{Y}$ reduces to the total number of left proper
focal points of the conjoined basis $Y$ in the interval $(a,b]$, while the dual
oscillation number $\Nsab{Y}$ reduces to the total number of right proper focal
points of\/ $Y$ in the interval $[a,b)$. These notions were introduced in
\cite{mW2007} and \cite{wK.rSH2012}. The main results in \cite{jvE2020a} and
\cite{jvE2020?c} provide, without assuming \eqref{E:LC} and without any majorant
condition on the involved coefficient matrices, the Sturmian type comparison and
separation theorems for conjoined bases of two possibly uncontrollable linear
Hamiltonian systems of the form \eqref{H} in terms of the oscillation numbers.

The main purpose of this paper is to extend the concepts of an~oscillation number
$\Nab{Y}$ and a~dual oscillation number $\Nsab{Y}$, following the definitions
presented in \cite{jvE2020,jvE2020a,jvE2020?c} via the comparative index, to an~arbitrary
continuous Lagrangian path $Y$ on $[a,b]$ and to connect these notions with the
Maslov index. As the main tools we employ the traditional theory of Lidskii angles
(or arguments) for symplectic matrices introduced in \cite{vbL1955,vaY1961,vaY1964}
and a~new connection of the Lidskii angles with the comparative index obtained
recently in \cite{pS.rSH2020?g}. More precisely, with a~given continuous Lagrangian
path $Y$ on $[a,b]$ we associate the special continuous symplectic and orthogonal
matrix $Z_Y(t)$ defined by
\begin{equation} \label{E:ZY.def}
  Z_Y(t):=\mmatrix{\J\hspt Y(t)\,K_Y(t) & Y(t)\,K_Y(t)}, \quad
  K_Y(t):=[Y^T\nhspt(t)\,Y(t)]^{-1/2}, \quad t\in[a,b],
\end{equation}
where $K_Y(t)>0$ represents a~normalization factor, and consider its continuous
Lidskii angles $\varphi_j(t)$ on $[a,b]$ for $j\in\{1,\dots,n\}$. We then express
(Theorem~\ref{T:osc.number.q.q*}) the oscillation number $\Nab{Y}$ in terms of the
cumulative changes of specific integers $q_j(t)$ at the endpoints of $[a,b]$. For
each $j\in\{1,\dots,n\}$ the integer $q_j(t)$ has the property that the Lidskii angle
$\varphi_j(t)$ belongs to the half-open interval
$[2\pi\hspt q_j(t),2\pi\hspt(q_j(t)+1))$. In a~similar way we express
the dual oscillation number $\Nsab{Y}$ in terms of the integer quantities $q_j^*(t)$
corresponding to the Lidskii angles $\varphi_j(t)$, which are located in the half-open
interval $(2\pi q_j^*(t),2\pi\hspt(q_j^*(t)+1)]$. Combining the latter two
results yields a~formula (Theorem~\ref{T:osc.dual.osc.number}) relating the
oscillation number and the dual oscillation number of one Lagrangian path on $[a,b]$.
We also prove a~special invariance property of the oscillation numbers under
a~symplectic orthogonal transformation (Theorem~\ref{T:osc.invariance.transf}).
We also present (under some monotonicity assumptions) formulas for the oscillation and
dual oscillation numbers of the Lagrangian path $Y$ in terms of the changes in the rank
of the first component $X$ of $Y$ (Theorem~\ref{T:monoton.osc}).

Using the Lidskii angles in our analysis of the oscillation numbers leads in a~natural
way to their connection with the Maslov index of two continuous Lagrangian paths $Y$
and $\hY$ on $[a,b]$, denoted by $\Masab{Y}{\hY}$. Here we use the analytic definition
of the Maslov index from \cite{bBB.kF1998}, see also
\cite{pH.sJ.bK2018,pH.yL.aS2017,pH.yL.aS2018}. We show (Theorem~\ref{T:osc.Maslov})
that the Maslov index of\/ $Y$ and $\hY$ on $[a,b]$ can be calculated from the changes
at the endpoints of $[a,b]$ of the Lidskii angles of the symplectic matrix
$Z_Y^{-1}(t)\hspt Z_{\hY}(t)$. This approach is known in
\cite[Definition~2.2]{yZ.lW.cZ2018} and \cite[Eqs.~(2.4)--(2.5)]{bBB.cZ2018}.
Consequently, the Maslov index  of $Y$ and $\hY$ on $[a,b]$ is equal to the
oscillation number of the transformed Lagrangian path $Z_Y^{-1}\hspt\hY$ on $[a,b]$, i.e.,
\begin{equation} \label{E:Maslov.osc.number}
  \Masab{Y}{\hY}=\Nab{Z_Y^{-1}\hspt\hY},
\end{equation}
where the symplectic and orthogonal matrix matrices $Z_Y(t)$ and $Z_{\hY}(t)$ are
defined according to \eqref{E:ZY.def}. Equivalently, the number in
\eqref{E:Maslov.osc.number} is equal to the oscillation number $\Nab{Z^{-1}\hY}$,
where $Z(t)$ is any continuous symplectic matrix, whose second block column is equal
to $Y(t)$. In particular, the oscillation number of\/ $Y$ on $[a,b]$ can be expressed
via the Maslov index as
\begin{equation} \label{E:osc.number.Maslov}
  \Nab{Y}=\Masab{E}{Y},
\end{equation}
where the matrix $E=\mmatrix{0 & I}^T$ represents the vertical Lagrangian plane. We also
discuss (Remark~\ref{R:dual.Maslov.index}) a~notion of the dual Maslov index, which is
related in analogous way to the dual oscillation number and to the dual comparative index,
and a~new monotonicity result for calculating the Maslov index (Theorem~\ref{T:monoton.osc.Maslov}).
Thus, we contribute in the spirit of \cite{pH.sJ.bK2018,pH.yL.aS2017,pH.yL.aS2018} to
a~detailed investigation of the Maslov index for Lagrangian paths in $\Rbb^{2n}$.

The main properties of the oscillation numbers obtained via the Lidskii angles allow to
derive Sturmian type comparison theorems for the oscillation numbers
(Theorem~\ref{T:compare.generalized}) of two continuous Lagrangian paths on $[a,b]$,
essentially extending the results in \cite[Theorem~4.4]{jvE2020a} and
\cite[Theorem~4.3]{jvE2020?c} to the context of continuous Lagrangian paths. These
formulas involve the comparative index or the dual comparative index of $Y(t)$ and $\hY(t)$
evaluated at the endpoints of the interval $[a,b]$. In addition, we derive Sturmian type
separation theorems (Theorems~\ref{T:osc.number.separ} and~\ref{T:distrib.osc.number}) for
specific continuous Lagrangian paths, which belong to the set of paths determined by a~given
continuous symplectic matrix $\Phi(t)$ on $[a,b]$. In this way we generalize the result in
\cite[Theorem~1.1]{pS.rSH2021} to the context of continuous Lagrangian paths.
The above mentioned comparison theorem for the oscillation numbers also yields
(Corollary~\ref{C:compare.Maslov}) the expression of the Maslov index
$\Masab{Y}{\hY}$ in terms of the two reference Maslov indices $\Masab{E}{\hY}$ and
$\Masab{E}{Y}$ and in terms of the comparative index of\/ $\hY$ and $Y$ evaluated at
the endpoints of the interval $[a,b]$.

We are convinced that the results in this paper contribute to the understanding of the
role of the comparative index in the oscillation theory of continuous Lagrangian paths.
We believe that further investigations in this direction will lead to substantial
advancements in several areas of theoretical mathematics, such as in the spectral
theory of linear Hamiltonian systems, where the monotonicity assumption on the spectral
parameter is dropped, in the oscillation theory on discrete time domains, or in the
theory of continuous symplectic matrices in general. In particular, the presented
methods are instrumental for future development of the oscillation numbers and the
Maslov index on unbounded intervals (including the properties of the rotation number).

The paper is organized as follows. In Section~\ref{S:Comparative} we recall the
definitions of the comparative index, the dual comparative index, and the Lidskii
angles of a~symplectic matrix. We also present their relationship based on special
argument functions derived in \cite{pS.rSH2020?g}. In Section~\ref{S:Osc.number} we
define the oscillation number and the dual oscillation number for a~continuous
Lagrangian path on $[a,b]$ and study their relationship with the Lidskii angles.
In Section~\ref{S:Maslov} we investigate connections of the oscillation number and the
dual oscillation number with the Maslov index. In Section~\ref{S:Osc.compare} we
present comparison and separation theorems for the oscillation numbers and the dual
oscillation numbers of two Lagrangian paths. Finally, in Section~\ref{S:Conclude} we
comment about the results of this paper and their future development.

\section{Comparative index and Lidskii angles} \label{S:Comparative}
In this section we recall the definitions of the comparative index, the dual
comparative index, and the Lidskii angles for a~symplectic matrix. We also present
formulas for the calculation of the comparative index and the dual comparative index
by means of the Lidskii angles of suitable symplectic matrices. Following
\cite{jvE2007,jvE2009a} or \cite[Section~3]{oD.jvE.rSH2019}, for two constant
matrices $Y,\hY\in\Rbb^{2n\times n}$ satisfying the properties in
\eqref{E:Y.Lagrangian.path.def} we define the {\em comparative index\/} $\mu(Y,\hY)$
and the {\em dual comparative index\/} $\mu^*(Y,\hY)$ by the formulas
\begin{equation} \label{E:comparative.index.def}
  \mu(Y,\hY):=\rank\M+\ind\P, \quad \mu^*(Y,\hY):=\rank\M+\ind\,(-\P),
\end{equation}
where the $n\times n$ matrices $\M$ and $\P$ are given by
\begin{equation} \label{E:MPV.matrices.def}
  \M:=(I-X^\dagger X)\,W(Y,\hY), \quad \P:=V\hspt [W(Y,\hY)]^TX^\dagger\hX V, \quad
  V:=I-\M^\dagger\M,
\end{equation}
and where $W(Y,\hY):=Y^T\!\J\hspt\hY$ is the {\em Wronskian\/} of $Y=(X^T,U^T)^T$ and
$\hY=(\hX^T,\hU^T)^T$, following the notation in \eqref{E:JYH.partition}. The dagger
in \eqref{E:MPV.matrices.def} denotes the Moore--Penrose pseudoinverse, see e.g.
\cite{aB.tneG2003,slC.cdM2009}. The notation $\ind\P$ means the number of negative
eigenvalues of the symmetric matrix $\P$. Note that $\ind\P+\ind\,(-\P)=\rank\P$.
For convenience we also define the constant $2n\times n$ matrix $E$ representing
the vertical Lagrangian plane, i.e.,
\begin{equation} \label{E:E.def}
  E:=\mmatrix{0 & I}^T.
\end{equation}
We will use the following basic invariant properties of the comparative indices defined in
\eqref{E:comparative.index.def}, see \cite[Properties~1--2, pg.~448]{jvE2009a} or
\cite[Theorem~3.5(i)--(iii)]{oD.jvE.rSH2019},
\begin{gather}
  \mu(YC_1,\hY C_2)=\mu(Y,\hY), \quad \mu^*(YC_1,\hY C_2)=\mu^*(Y,\hY), \quad
    \det C_1\neq0, \ \det C_2\neq0, \label{prop1} \\
  \mu(LY,L\hY)=\mu(Y,\hY), \quad \mu^*(LY,L\hY)=\mu^*(Y,\hY), \quad
    L=\mmatrix{P & 0 \\ K & P^{T-1}}, \  L^T\!\nhspt\J L=\J, \label{prop2} \\
  \mu(Y,\hY)=\mu^*(Z^{-1}E,Z^{-1}\hY), \quad \mu^*(Y,\hY)=\mu(Z^{-1}E,Z^{-1}\hY),
    \quad Z^T\!\nhspt\J\nhspt Z=\J, \  Z\nhspt E=Y. \label{prop3}
\end{gather}

Next we recall the definition of Lidskii angles. Consider a~real symplectic matrix
\begin{equation} \label{E:S.def}
  S=\mmatrix{S_{11} & S_{12} \\ S_{21} & S_{22}}, \quad S_{ij}\in\Rbb^{n\times n},
  \quad S^T\!\J S=\J.
\end{equation}
We define, according to \cite{vbL1955,vaY1964,fvA1964} or
\cite[Section~3.1]{rJ.sN.cN.rO2017}, the complex $n\times n$ matrix
\begin{equation} \label{E:W.def}
  W_{\nhspt S}:=(S_{11}-iS_{12})^{-1}\hspt(S_{11}+iS_{12}),
\end{equation}
where $i$ is the imaginary unit ($i^{\hspt2}=-1$). The matrix $W_{\nhspt S}$ is
well-defined, symmetric, and unitary, and in particular its eigenvalues $w_j$ lie
on the unit circle $\Ubb$ in the complex plane. Hence, $w_j=\exp\hspt(i\varphi_j)$
with the real arguments $\varphi_j$. The numbers $\varphi_j$ for $j\in\{1,\dots,n\}$
are called the {\em Lidskii angles\/} corresponding to the symplectic matrix $S$. The
angles $\varphi_j$ are thus defined uniquely up to an~additive term $2\pi m_j$ for
$m_j\in\Zbb$ and
\begin{equation} \label{E:arg.WS}
  \sum_{j=1}^n\varphi_j=\arg\det W_{\nhspt S}
  =\big(2\hspt\arg\det\hspt(S_{11}+iS_{12})\big) \!\!\!\!\mod2\pi.
\end{equation}
In view of \eqref{E:arg.WS} the Lidskii angles $\varphi_j$ satisfy
\begin{equation} \label{E:Arg3}
  \frac{1}{2}\,\sum_{j=1}^n\varphi_j=\Arg_3(S) \!\!\!\!\mod\pi,
\end{equation}
where $\Arg_3(S)$ is one of the arguments of the symplectic matrix $S$ considered
by Yakubovich in \cite[pg.~263]{vaY1961}, see also
\cite[Lemma~5.6]{rJ.rO.sN.cN.rF2016}, namely
\begin{equation} \label{E:Arg3.def}
  \Arg_3(S):=\arg\det\hspt(S_{11}+iS_{12})
  =\arg\det\!\bigg(\!\!\mmatrix{I & 0} S\mmatrix{I \\ iI}\!\!\bigg).
\end{equation}

In \cite{pS.rSH2020?g} we introduced the following argument function,
which is motivated by \eqref{E:Arg3}. Given a~real interval $\I$ and a~real symplectic
matrix $S$, we consider the Lidskii angles $\varphi_j$ of $S$, which belong to the
interval $\I$ for all $j\in\{1,\dots,n\}$. Then we set
\begin{equation} \label{E:Arg.I.def}
  \ArgI(S\hspt):=\frac{1}{2}\,\sum_{j=1}^n\varphi_j
  \qquad\text{with $\varphi_j\in\I$ for all $j\in\{1,\dots,n\}$}.
\end{equation}
We will use the above argument function $\ArgI$ with the particular choice of the
half-open intervals $\I=[2\pi q,2\pi\hspt(q+1))$ and
$\I=(2\pi q,2\pi\hspt(q+1)]$ for $q\in\Zbb$.

\begin{Remark} \label{R:Lidskii.S12}
  It is known that the number of the Lidskii angles $\varphi_j$ of $S$,
  which are integer multiples of $2\pi$, is equal to the defect of the block $S_{12}$
  in \eqref{E:S.def}, see e.g. \cite[Proposition~2.5]{rF.rJ.sN.cN2011} or
  \cite[Eq.~(4.10)]{pS.rSH2020?e}.
\end{Remark}

The main result in \cite{pS.rSH2020?g} provides a~connection between the
comparative index (or the dual comparative index) and the Lidskii angles. The
connection is based on the symplectic and orthogonal matrices of the form
\eqref{E:ZY.def}. More precisely, for a~real constant $2n\times n$ matrix $Y$
satisfying \eqref{E:Y.Lagrangian.path.def} we define the constant $2n\times2n$ matrix
\begin{equation} \label{E:ZY.def.const}
  Z_Y:=\mmatrix{\J\hspt Y\nhspt K_Y & Y\nhspt K_Y}, \quad
  K_Y:=(Y^TY)^{-1/2},
\end{equation}
where the symmetric matrix $K_Y$ obeys the condition $K_Y>0$. Then the matrix $Z_Y$
is symplectic and orthogonal, i.e., $Z_Y^T\J Z_Y=\J$ and $Z_Y^TZ_Y=I$. This implies
that the matrices $Z_Y$ and $\J$ commute. Note also that $Z_E=I$.

\begin{Remark} \label{R:S=ZL}
  Observe that for an~arbitrary symplectic matrix $S$ we have
  \begin{equation} \label{E:SZYL}
    S=Z_{SE}\hspt L,
  \end{equation}
  where $L=Z_{SE}^{-1}\hspt S=Z_{SE}^T\hspt S$ is a~symplectic lower block triangular matrix
  in the form given by \eqref{prop2}. This fact follows from equation \eqref{E:ZY.def.const}
  with $Y:=SE$. Indeed, the right upper block of the matrix $L$ is equal to
  $(I\ 0)\hspt LE=(I\ 0)\hspt Z_{SE}^{T}SE=-K_Y(Y^T\!\J Y)=0$. Moreover, if $S$ is a~symplectic and
  orthogonal matrix, then in \eqref{E:SZYL} we have $L=I$ and hence by \eqref{E:ZY.def.const}
  we get $Z_{SY}=SZ_Y$ and $Z_{S\nhspt E}=S$.
\end{Remark}

The following result is derived in \cite[Theorem~1.1]{pS.rSH2020?g}.

\begin{Proposition}[Comparative index and Lidskii angles] \label{P:main.CI.Lidskii}
  Let $Y$ and $\hY$ be real constant $2n\times n$ matrices satisfying
  \eqref{E:Y.Lagrangian.path.def} and define the symplectic matrices $Z_Y$ and $Z_{\hY}$
  by \eqref{E:ZY.def.const}. Then the comparative index and the dual comparative
  index defined in \eqref{E:comparative.index.def} with \eqref{E:MPV.matrices.def}
  satisfy
  \begin{align}
    \mu(Y,\hY) &= \frac{1}{\pi}\,\Big\{
      \ArgL(Z_{\hY})-\ArgL(Z_Y)+\ArgL(Z_{\hY}^{-1}Z_Y) \Big\}, \label{E:mu.Lidskii} \\
    \mu^*(Y,\hY) &= n-\frac{1}{\pi}\,\Big\{
      \ArgR(Z_{\hY})-\ArgR(Z_Y)+\ArgR(Z_{\hY}^{-1}Z_Y) \Big\}, \label{E:mu*.Lidskii}
  \end{align}
  where $\ArgI$ is the special argument function defined in \eqref{E:Arg.I.def}.
\end{Proposition}

In the last comment in this section we discuss the continuity property of the
Lidskii angles.

\begin{Remark} \label{R:Lidskii.continuous}
  Consider a~continuous Lagrangian path $Y$ on the interval $[a,b]$ and define the
  symplectic and orthogonal matrix $Z_Y(t)$ on $[a,b]$ by \eqref{E:ZY.def}. Note that
  this matrix corresponds to the matrix $Z_{Y(t)}$ in \eqref{E:ZY.def.const}. Then
  the matrix $Z_Y(t)$ is continuous on $[a,b]$, which implies that the corresponding
  matrix $W_{Z_Y(t)}$ defined through \eqref{E:W.def} is also continuous on $[a,b]$.
  Therefore, the Lidskii angles $\varphi_j(t)$ of the symplectic matrix $Z_Y(t)$
  can be chosen to be continuous functions on $[a,b]$, see also the proof of
  \cite[Theorem~3.6]{rJ.sN.cN.rO2017}. This fact is utilized in
  Definition~\ref{D:Lidskii.Lagrangian.path} below. In the next sections we
  will always work with such continuous Lidskii angles $\varphi_j(t)$ on $[a,b]$
  for all $j\in\{1,\dots,n\}$.
\end{Remark}

\section{Oscillation numbers for Lagrangian paths} \label{S:Osc.number}
In this section we generalize the notion of the oscillation number and the dual oscillation number
from \cite[pp.~17--18]{jvE2019}, \cite[pp.~311--312]{jvE2020a}, and \cite[Definition~3.1]{jvE2020?c},
which are defined through the comparative index and the dual comparative index, in the context of
continuous Lagrangian paths on $[a,b]$ and study their properties in terms of the Lidskii angles.
For a~given continuous Lagrangian path $Y$ on $[a,b]$, let $D:=\{a=t_0<t_1<\dots<t_{p-1}<t_p=b\}$
be a~finite partition of the interval $[a,b]$ and let $\{R_k(t)\}_{k=0}^{p-1}$ be a~finite
system of real symplectic matrices on $[a,b]$ such that for every $k\in\{0,\dots,p-1\}$ the matrix
$R_k(t)$ is continuous on the interval $[t_k,t_{k+1}]$ and the transformed Lagrangian path
$\tY_k=(\tX_k^T,\tU_k^T)^T:=R_k^{-1}Y$ satisfies
\begin{equation} \label{E:tY.def.transformed.new}
  \left. \begin{array}{c}
    \rank\tX_k(t)=\rank W(R_k(t)E,Y(t))=\const, \quad
      \rank\!\big((I\ 0)\hspt R_k(t)E\big)=\const, \\[1mm]
    t\in[t_k,t_{k+1}], \quad k\in\{0,1,\dots,p-1\}.
  \end{array} \!\right\}
\end{equation}

For a~continuous Lagrangian path $Y$ on $[a,b]$ and a~partition $D=\{t_k\}_{k=0}^p$ of
$[a,b]$ with the system of symplectic matrices $R:=\{R_k(t)\}_{k=0}^{p-1}$ satisfying
\eqref{E:tY.def.transformed.new} we define the following quantities, compare with
\cite[Eq.~(2.6)]{jvE2019}, \cite[Eq.~(3.4)]{jvE2020a}, and \cite[Eq.~(3.1)]{jvE2020?c},
\begin{align}
  \Nab{Y,D,R} &:= \sum_{k=0}^{p-1}\mu(Y(t),R_k(t)E)\hspt\big|_{t_k}^{t_{k+1}},
    \label{E:Nab.D.R.def} \\
  \Nsab{Y,D,R} &:= \sum_{k=0}^{p-1}\mu^*(Y(t),R_k(t)E)\hspt\big|_{t_{k+1}}^{t_{k}},
    \label{E:Nab*.D.R.def}
\end{align}
where we use the comparative index in \eqref{E:Nab.D.R.def} and the dual comparative
index in \eqref{E:Nab*.D.R.def}. Here for a~function $f:[a,b]\to\Rbb$ we use the standard notation
\begin{equation*}
  f(t)\big|_{\tau_1}^{\tau_2}:=f(\tau_2)-f(\tau_1),
  \quad \tau_1,\tau_2\in[a,b].
\end{equation*}
Based on \eqref{E:tY.def.transformed.new} and property \eqref{prop3} of the comparative index we derive the connections
\begin{equation} \label{f1}
  \left. \begin{array}{rl}
    \Nab{Y,D,R} &\!\!\!\!= -\Nsab{Z^{-1}E,D,\tR}, \\[1mm]
    \Nsab{Y,D,R} &\!\!\!\!= -\Nab{Z^{-1}E,D,\tR}.
  \end{array} \!\right\}
\end{equation}
Here $Z(t)$ is any continuous symplectic matrix with $Z(t)\hspt E=Y(t)$ on $[a,b]$ and
$\tR=\{\tR_k(t)\}_{k=0}^{p-1}$ is the transformed system of symplectic matrix-valued functions
$\tR_k(t):=Z^{-1}(t)\hspt R_k(t)$ obtained from the original system $R=\{R_k(t)\}_{k=0}^{p-1}$,
compare with \cite[Proposition 3.3(vi)]{jvE2020?c}. Indeed, by applying \eqref{prop3} we rewrite
equations \eqref{E:Nab.D.R.def} and \eqref{E:Nab*.D.R.def} as
\begin{align}
  \Nab{Y,D,R} &= \sum_{k=0}^{p-1}\mu^*(Z^{-1}(t)\hspt E,Z^{-1}(t)\hspt R_k(t)E)\hspt\big|_{t_k}^{t_{k+1}},
    \label{E:Nab.D.R.def1} \\
  \Nsab{Y,D,R} &= \sum_{k=0}^{p-1}\mu(Z^{-1}(t)\hspt E,Z^{-1}(t)\hspt R_k(t)E)\hspt\big|_{t_{k+1}}^{t_{k}},
    \label{E:Nab*.D.R.def1}
\end{align}
where the pair of matrices $Z^{-1}(t)\hspt E$ and $\tilde R_k(t)$ satisfies the properties in assumption
\eqref{E:tY.def.transformed.new}, replacing respectively the matrices $Y(t)$ and $R_k(t)$ therein.
Incorporating the order of the substitutions in \eqref{E:Nab.D.R.def} and \eqref{E:Nab*.D.R.def}, we derive
the result in \eqref{f1} from equations \eqref{E:Nab.D.R.def1} and \eqref{E:Nab*.D.R.def1}.

\begin{Remark} \label{R:partition}
  \par(i) In \cite{jvE2019} we considered a~partition $D$ and a~system $R=\{R_k\}_{k=0}^{p-1}$ of constant
  symplectic matrices such that instead of \eqref{E:tY.def.transformed.new} we have that
  \begin{equation} \label{E:tY.def.transformed}
    \text{$\tX_k(t)$ is invertible on the interval $[t_k,t_{k+1}]$}.
  \end{equation}
  For the special case of
  \begin{equation*}
    R_k:=R_{\alpha_k}
    =\mmatrix{(\cos\alpha_k)\hspt I & (\sin\alpha_k)\hspt I \\ -(\sin\alpha_k)\hspt I & (\cos\alpha_k)\hspt I}
  \end{equation*}
  a~direct proof of the existence of a~partition $D$ with property \eqref{E:tY.def.transformed} can be found
  in \cite[pg.41]{aaA2011} and \cite[Lemma 3.5]{jvE2020a}. Using the representation
  $\tX_k(t)=-W(R_kE,Y(t))$, condition \eqref{E:tY.def.transformed} can be rewritten in the form
  \begin{equation} \label{EJ:tY.def.transformed}
    \rank W(R_kE,Y(t))=n,\, \quad t\in[t_k,t_{k+1}].
  \end{equation}
  Condition \eqref{EJ:tY.def.transformed} then means that the Lagrangian path $Y$ is {\it transversal\/}
  to the Lagrangian plane $R_kE$ on the interval $[t_k,t_{k+1}]$.
  \par(ii) In \cite[Eq.~(3.13)]{jvE2020a} we introduced an~equivalent definition of the oscillation
  numbers from \cite{jvE2019} for conjoined bases $Y$ of system \eqref{H}. Namely, consider a~partition
  $a=s_0<s_1<\dots<s_{r-1}<s_r=b$ of $[a,b]$ such that for any $\ell\in\{0,\dots,r-1\}$ there exists
  a~conjoined basis $Y_\ell$ of system \eqref{H} with the nonsingular upper block $X_\ell(t)$ on
  $[s_\ell,s_{\ell+1}]$, i.e.,
  \begin{equation}\label{nonsingalon2}
    \det X_\ell(t)\neq 0, \quad t\in[s_\ell,s_{\ell+1}], \quad \ell\in\{0,\dots,r-1\}.
  \end{equation}
  The oscillation numbers in \cite[Eq.~(3.13)]{jvE2020a} are defined by \eqref{E:Nab.D.R.def} with
  $R_\ell(t):=Z_\ell(t)$, where $Z_\ell(t)$ is a~symplectic fundamental matrix of system \eqref{H}
  such that $Z_\ell(t)\hspt E=Y_\ell(t)$ on $[s_\ell,s_{\ell+1}]$. Note that for this choice of
  $R_\ell(t)$ all conditions in \eqref{E:tY.def.transformed.new} are satisfied, in particular $Y$ and
  $R_\ell\hspt E=Y_\ell$ are conjoined bases of \eqref{H} and then the first condition in
  \eqref{E:tY.def.transformed.new} automatically holds. Hence, we can see that condition
  \eqref{E:tY.def.transformed.new} is a~proper generalization of conditions
  \eqref{E:tY.def.transformed} and \eqref{nonsingalon2}.
  \par(iii) For an~arbitrary continuous Lagrangian path $Y$ on $[a,b]$ one can present an~analog of the
  definition from part (ii) by putting $R_k(t):=Z_{Y(t)}C_k$, where $Z_{Y(t)}$ are symplectic and orthogonal
  matrices given by \eqref{E:ZY.def} and $C_k$ for $k\in\{0,1,\dots,p-1\}$ are constant symplectic matrices
  such that the upper block of $R_k(t)E$ is nonsingular for all $t\in[t_k,t_{k+1}]$ or, more generally,
  it has constant rank on $[t_k,t_{k+1}]$.
  \par(iv) For the subsequent proofs it is important that condition \eqref{E:tY.def.transformed.new}
  remains valid for the matrices $Z_{R_kE}E=R_kEK_{R_kE}$ and $Z_{Y(t)}E=Y(t)\hspt K_{Y(t)}$, which
  are associated with $R_kE$ and $Y(t)$ via \eqref{E:ZY.def.const}. Indeed, according to the definition in
  \eqref{E:ZY.def} the matrices $K_{R_kE}$ and $K_{Y(t)}$ are nonsingular and the values
  $\rank W(Z_{R_kE}\hspt E,Z_{Y(t)}E)=\rank W(R_kE,Y(t))$ and
  $\rank((I\,\ 0)\hspt R_k(t)\hspt EK_{R_kE})$ are then constant on $[t_k,t_{k+1}]$.
\end{Remark}

For the cases described in Remark~\ref{R:partition}(i),~(ii) the numbers defined in \eqref{E:Nab.D.R.def}
and \eqref{E:Nab*.D.R.def} are invariant with respect to the partition $D$ of $[a,b]$ and with respect
to the special choice of the system $R$ of symplectic matrices satisfying condition
\eqref{E:tY.def.transformed} or \eqref{nonsingalon2}, see \cite[Lemma~2.2]{jvE2020a} and
\cite[Proposition~3.3(i)]{jvE2020?c}. One can prove the same invariant properties for the general case
\eqref{E:tY.def.transformed.new}, based only on the properties of the comparative index and the dual
comparative index. Note that the same will follow from our Theorem~\ref{T:osc.number.q.q*} below (see
Remark~\ref{R:osc.number.independent}). The invariant properties of the numbers in \eqref{E:Nab.D.R.def}
and \eqref{E:Nab*.D.R.def} then justify the following definition.

\begin{Definition}[Oscillation numbers] \label{D:osc.number}
  Let $Y:[a,b]\to\Rbb^{2n\times n}$ be a~continuous Lagrangian path. The quantity
  defined in \eqref{E:Nab.D.R.def} is called the {\em oscillation number\/} of\/ $Y$
  on $[a,b]$ and it is denoted by $\Nab{Y}$. The quantity defined in
  \eqref{E:Nab*.D.R.def} is called the {\em dual oscillation number\/} of\/ $Y$ on
  $[a,b]$ and it is denoted by $\Nsab{Y}$.
\end{Definition}

\begin{Remark} \label{R:osc.number.multiple}
  (i) Since by \eqref{prop1} the comparative index and the dual comparative index are invariant under
  the multiplication of its arguments by invertible $n\times n$ matrices from the right, it follows that
  the same property holds for the oscillation number and the dual oscillation number. More precisely,
  if $Y$ is a~continuous Lagrangian path on $[a,b]$ and $C:[a,b]\to\Rbb^{n\times n}$
  is a~continuous invertible matrix function on $[a,b]$, then $Y\nhspt C$ is also
  a~continuous Lagrangian path on $[a,b]$ with
  \begin{equation} \label{E:osc.number.multiple}
    \Nab{Y\nhspt C}=\Nab{Y}, \quad \Nsab{Y\nhspt C}=\Nsab{Y}.
  \end{equation}
  \par(ii) The definition of the oscillation number and the dual oscillation number
  through \eqref{E:Nab.D.R.def} and \eqref{E:Nab*.D.R.def} yields their additivity
  with respect to the base interval, see also \cite[Remark~3.7]{jvE2020a}. Namely,
  for any point $c\in(a,b)$ we have
  \begin{equation} \label{E:osc.number.additive}
    \Nab{Y}=\N(Y,[a,c])+\N(Y,[c,b]), \quad
    \Nsab{Y}=\N^*(Y,[a,c])+\N^*(Y,[c,b]).
  \end{equation}
\end{Remark}

We will now present the main results of this section, which use the following
terminology based on Remark~\ref{R:Lidskii.continuous}.

\begin{Definition}[Lidskii angles of Lagrangian path] \label{D:Lidskii.Lagrangian.path}
  Let $Y$ be a~continuous Lagrangian path on $[a,b]$ and let $Z_Y(t)$ be the associated symplectic
  (and orthogonal) matrix defined in \eqref{E:ZY.def}. Then the continuous Lidskii angles
  $\varphi_j(t)$ for $j\in\{1,\dots,n\}$ and $t\in[a,b]$ of $Z_Y(t)$ are called the {\em Lidskii
  angles of the Lagrangian path\/} $Y$.
\end{Definition}

With the above terminology, for fixed continuous branches of the Lidskii angles
$\varphi_j(t)$ we consider the uniquely defined integers $q_j(t)$ and $q_j^*(t)$,
which satisfy the properties
\begin{align}
  \varphi_j(t)\in[2\pi\hspt q_j(t),2\pi\hspt(q_j(t)+1)),
    \quad t\in[a,b], \quad j\in\{1,\dots,n\}, \label{E:qjt.def} \\
  \varphi_j(t)\in(2\pi\hspt q_j^*(t),2\pi\hspt(q_j^*(t)+1)],
    \quad t\in[a,b], \quad j\in\{1,\dots,n\}, \label{E:qjt*.def}
\end{align}
that is, the Lidskii angle $\varphi_j(t)$ of $Y$ belongs to the indicated half-open
interval of the length $2\pi$. This means that the integers $q_j(t)$ and $q_j^*(t)$
are given by
\begin{equation} \label{E:qj.qj*.floor.ceiling}
  q_j(t)=\Big\lfloor\frac{\varphi_j(t)}{2\pi}\Big\rfloor, \quad
  q_j^*(t)=\Big\lceil\frac{\varphi_j(t)}{2\pi}\Big\rceil-1, \quad
  t\in[a,b],
\end{equation}
where for $x\in\Rbb$ the notation $\lfloor x\rfloor$ and $\lceil x\rceil$ stand for
the greatest integer which is smaller or equal to $x$ (the floor function) and for
the smallest integer which is greater or equal to $x$ (the ceiling function).
It follows from this definition that $q_j(t)=q_j^*(t)$
if and only if the angle $\varphi_j(t)$ belongs to the interior of the intervals in
\eqref{E:qjt.def} or \eqref{E:qjt*.def}, while $q_j(t)=q_j^*(t)+1$ holds when the
angle $\varphi_j(t)$ is an~integer multiple of $2\pi$. This means in view of
Remark~\ref{R:Lidskii.S12} that
\begin{equation} \label{E:qj.qj*}
  \sum_{j=1}^nq_j(t)=\defect X(t)+\sum_{j=1}^nq_j^*(t) \quad\text{for all } t\in[a,b].
\end{equation}
Observe that for the argument function in \eqref{E:Arg3} we then have
\begin{align}
  \Arg_3(Z_Y(t))\big|_{\tau_1}^{\tau_2}
    =\frac{1}{2}\,\sum_{j=1}^n\varphi_j(t)\big|_{\tau_1}^{\tau_2}
  &= \ArgL(Z_Y(t))\big|_{\tau_1}^{\tau_2}
    +\pi\sum_{j=1}^nq_j(t)\big|_{\tau_1}^{\tau_2}, \label{EJ:Arg3.ArgL} \\
  \Arg_3(Z_Y(t))\big|_{\tau_1}^{\tau_2}
    =\frac{1}{2}\,\sum_{j=1}^n\varphi_j(t)\big|_{\tau_1}^{\tau_2}
  &= \ArgR(Z_Y(t))\big|_{\tau_1}^{\tau_2}
    +\pi\sum_{j=1}^nq_j^*(t)\big|_{\tau_1}^{\tau_2} \label{EJ:Arg3.ArgR}
\end{align}
for all $\tau_1,\tau_2\in[a,b]$ with $\tau_1<\tau_2$.
We have the following important property of the integers $q_j(t)$ and $q_j^*(t)$ in \eqref{E:qjt.def}
and \eqref{E:qjt*.def}.

\begin{Lemma}\label{L:cons.rank}
  Let $Y$ be a~continuous Lagrangian path on $[a,b]$ with the associated Lidskii angles $\varphi_j(t)$
  on $[a,b]$ for $j\in\{1,\dots,n\}$ according to Definition~\ref{D:Lidskii.Lagrangian.path}. If the upper
  block $X(t)$ of\/ $Y(t)$ has constant rank on $[\tau_1,\tau_2]\subseteq[a,b]$, then
  \begin{equation} \label{E:osc.number.separ.hlp2}
    \text{$q_j(t)\equiv q_j$ and $q_j^{\hspt*}(t)\equiv q_j^{\hspt*}$ are constant on $[\tau_1,\tau_2]$
    for all $j\in\{1,\dots,n\}$},
  \end{equation}
  where $q_j(t)$ and $q_j^{\hspt*}(t)$ are given by \eqref{E:qjt.def} and \eqref{E:qjt*.def}.
 \end{Lemma}
\begin{proof}
  According to Remark~\ref{R:Lidskii.S12} and using the continuity of $\varphi_j(t)$ on $[a,b]$ it follows
  that each angle $\varphi_j(t)$ either remains in the open interval
  $(2\pi\hspt q_j,2\pi\hspt(q_j+1))=(2\pi\hspt q_j^{\hspt*},2\pi\hspt(q_j^{\hspt*}+1))$ on $[a,b]$, or it
  is constant on the interval $[\tau_1,\tau_2]$ with the value
  $\varphi_j(t)\equiv2\pi\hspt q_j=2\pi\hspt(q_j^{\hspt*}+1)$. Hence, the property in
  \eqref{E:osc.number.separ.hlp2} holds.
\end{proof}

Combining Proposition~\ref{P:main.CI.Lidskii} with equations \eqref{EJ:Arg3.ArgL} and
\eqref{EJ:Arg3.ArgR} we derive the following main property of the comparative index
for two continuous Lagrangian paths.

\begin{Proposition}[Comparative index for continuous Lagrangian paths]
  \label{PJ:main.CI.Lidskii}
  Let $Y$ and $\hY$ be continuous Lagrangian paths on $[a,b]$ with the associated
  Lidskii angles $\varphi_j(t)$ and $\hat\varphi_j(t)$ on $[a,b]$ for
  $j\in\{1,\dots,n\}$ according to Definition~\ref{D:Lidskii.Lagrangian.path}.
  Consider the continuous symplectic matrix
  \begin{equation} \label{EJ:St.ZY.ZhY.def}
    \tilde S(t):=Z_{\hY}^{-1}(t)\,Z_{Y}(t)=Z_{\hY}^T\nhspt(t)\,Z_{Y}(t), \quad t\in[a,b],
  \end{equation}
  and its continuous Lidskii angles $\tilde\varphi_j(t)$, where $Z_Y(t)$ and
  $Z_{\hY}(t)$ are the symplectic and orthogonal matrices associated with $Y(t)$ and
  $\hY(t)$ through \eqref{E:ZY.def}. Given the integers $q_j(t)$, $\hat q_j(t)$,
  $\tilde q_j(t)$ and $q_j^*(t)$, ${\hat q}_j^*(t)$, ${\tilde q}_j^*(t)$
  associated through \eqref{E:qjt.def} and \eqref{E:qjt*.def} with the angles
  $\varphi_j(t)$, $\hat\varphi_j(t)$, $\tilde\varphi_j(t)$, then for any two points
  $\tau_1,\tau_2\in[a,b]$ the comparative index and the dual comparative
  index defined in \eqref{E:comparative.index.def} with \eqref{E:MPV.matrices.def}
  satisfy
  \begin{align}
    \mu(Y(t),\hY(t))\big|_{\tau_1}^{\tau_2} &=
      \sum_{j=1}^n \big(q_j(t)-{\hat q}_j(t)-\tilde q_j(t)\big)
      \big|_{\tau_1}^{\tau_2}, \label{EJ:mu.Lidskii} \\
    \mu^*(Y(t),\hY(t))\big|_{\tau_1}^{\tau_2} &=
      -\sum_{j=1}^n\big(q_j^*(t)-\hat q_j^*(t)-\tilde q_j^*(t)\big)
      \big|_{\tau_1}^{\tau_2}.  \label{EJ:mu*.Lidskii}
  \end{align}
\end{Proposition}
\begin{proof}
  According to \cite[pg.~163]{imG.vbL1958}, for any symplectic and orthogonal
  matrices $Z(t)$ and $\hZ(t)$ we have the multiplicative property
  \begin{equation*} \label{EJE:Arg3}
    \mmatrix{I & 0}\nhspt\hZ^{-1}(t)\hspt Z(t) \mmatrix{I \\ iI}
    =\mmatrix{I & 0}\nhspt\hZ^{-1}(t) \mmatrix{I \\ iI}\!\mmatrix{I & 0}\!Z(t)
    \mmatrix{I \\ iI}.
  \end{equation*}
  Then the argument function $\Arg_3$ defined by \eqref{E:Arg3.def} for the matrix
  $\hZ^{-1}(t)Z(t)$ is the sum of the argument functions of $\hZ^{-1}(t)$ and $Z(t)$,
  and hence
  \begin{equation}\label{Jul15.12}
    \Arg_3(\hZ(t))\big|_{\tau_1}^{\tau_2}-\Arg_3(Z(t))\big|_{\tau_1}^{\tau_2}
    +\Arg_3(\hZ^{-1}(t)\hspt Z(t))\big|_{\tau_1}^{\tau_2}=0.
  \end{equation}
  For any continuous Lagrangian paths $Y$ and $\hY$ on $[a,b]$ we then have by
  \eqref{E:mu.Lidskii} and \eqref{EJ:Arg3.ArgL} that
  \begin{align*}
    \pi\,\mu(Y(t),\hY(t))\big|_{\tau_1}^{\tau_2} &\overset{\eqref{E:mu.Lidskii}}{=}
      \ArgL(Z_{\hY}(t))\big|_{\tau_1}^{\tau_2}-\ArgL(Z_Y(t))\big|_{\tau_1}^{\tau_2}
      +\ArgL(Z_{\hY}^{-1}(t)\hspt Z_Y(t))\big|_{\tau_1}^{\tau_2} \\
    &\overset{\eqref{EJ:Arg3.ArgL}}{=}
      \Arg_3(Z_{\hY}(t))\big|_{\tau_1}^{\tau_2}-\Arg_3(Z_Y(t))\big|_{\tau_1}^{\tau_2}
      +\Arg_3(Z_{\hY}^{-1}(t)\hspt Z_Y(t))\big|_{\tau_1}^{\tau_2} \\
    &\hspace*{10mm}
      +\pi \sum_{j=1}^n\big(q_j(t)-{\hat q}_j(t)-{\tilde q}_j(t)
      \big)\big|_{\tau_1}^{\tau_2}.
  \end{align*}
  Then by \eqref{Jul15.12} we derive the equality in \eqref{EJ:mu.Lidskii}.
  The proof of equality \eqref{EJ:mu*.Lidskii} is similar and it follows from
  \eqref{E:mu*.Lidskii}, \eqref{EJ:Arg3.ArgR}, and \eqref{Jul15.12}.
\end{proof}

\begin{Corollary} \label{C:const.ranks}
  Under the assumptions and the notation of Proposition~\ref{PJ:main.CI.Lidskii} we have the following
  implications. If\/ $\rank \hX(t)$ is constant on $[\tau_1,\tau_2]$, then
  \begin{equation} \label{constRhatX}
    \mu(Y(t),\hY(t))\big|_{\tau_1}^{\tau_2}
      =\sum_{j=1}^n \big(q_j(t)-\tilde q_j(t)\big)\big|_{\tau_1}^{\tau_2}, \quad
    \mu^*(Y(t),\hY(t))\big|_{\tau_1}^{\tau_2}
      =\sum_{j=1}^n \big(q_j^*(t)-\tilde q_j^*(t)\big)\big|_{\tau_2}^{\tau_1},
  \end{equation}
  while if\/ $\rank W(\hY(t),Y(t))$ is constant on $[\tau_1,\tau_2]$, then
  \begin{equation} \label{constRW}
    \mu(Y(t),\hY(t))\big|_{\tau_1}^{\tau_2}
      =\sum_{j=1}^n \big(q_j(t)-{\hat q}_j(t)\big)\big|_{\tau_1}^{\tau_2},  \quad
    \mu^*(Y(t),\hY(t))\big|_{\tau_1}^{\tau_2}
      =\sum_{j=1}^n \big(q_j^*(t)-{\hat q}_j^*(t)\big)\big|_{\tau_2}^{\tau_1}.
  \end{equation}
  In particular, if both\/ $\rank \hX(t)$ and\/ $\rank W(\hY(t),Y(t))$ are constant on
  $[\tau_1,\tau_2]$, then
  \begin{equation} \label{constRhatXW}
    \mu(Y(t),\hY(t))\big|_{\tau_1}^{\tau_2}=\sum_{j=1}^n q_j(t)\big|_{\tau_1}^{\tau_2}, \quad
    \mu^*(Y(t),\hY(t))\big|_{\tau_1}^{\tau_2}=\sum_{j=1}^n q_j^*(t)\big|_{\tau_2}^{\tau_1}.
  \end{equation}
\end{Corollary}
\begin{proof}
  Under the stated assumptions, the results in \eqref{constRhatX}, \eqref{constRW}, and \eqref{constRhatXW}
  follow from Proposition~\ref{PJ:main.CI.Lidskii} and Lemma~\ref{L:cons.rank}.
\end{proof}

Our main result presented below shows that the oscillation number and the dual oscillation number of\/
$Y$ on $[a,b]$ count the cumulative change in the differences of the integers $q_j(t)$ and of the
integers $q_j^*(t)$ at the endpoints of the interval $[a,b]$.

\begin{Theorem} \label{T:osc.number.q.q*}
  Let $Y$ be a~continuous Lagrangian path on $[a,b]$ with the associated Lidskii angles $\varphi_j(t)$
  for $j\in\{1,\dots,n\}$ on $[a,b]$ according to Definition~\ref{D:Lidskii.Lagrangian.path}. Given the
  integers $q_j(t)$ and $q_j^*(t)$ satisfying conditions \eqref{E:qjt.def} and \eqref{E:qjt*.def}, then
  the oscillation number of\/ $Y$ on $[a,b]$ and the dual oscillation number of\/ $Y$
  on $[a,b]$ satisfy
  \begin{align}
    \Nab{Y} &= \sum_{j=1}^n\big(q_j(b)-q_j(a)\big), \label{E:osc.number.q} \\
    \Nsab{Y} &= \sum_{j=1}^n\big(q_j^*(b)-q_j^*(a)\big). \label{E:osc.number.q*}
  \end{align}
\end{Theorem}
\begin{proof}
  According to Definition~\ref{D:osc.number} we choose a~partition $D=\{t_k\}_{k=0}^p$ of the interval
  $[a,b]$ and a~system $R=\{R_k(t)\}_{k=0}^{p-1}$ of symplectic matrix-valued functions such that
  \eqref{E:tY.def.transformed.new} holds. By \eqref{constRhatXW} in Corollary~\ref{C:const.ranks} applied
  to $\hY(t):=R_k(t)E$ and $\tau_1:=t_k$, $\tau_2:=t_{k+1}$ for $k\in\{0,1,\dots,p-1\}$ we have
  for the variations of the comparative indices in \eqref{E:Nab.D.R.def} and \eqref{E:Nab*.D.R.def} that
  \begin{equation} \label{E:osc.number.q.hlp0}
    \mu(Y(t),R_k(t)E)\big|_{t_k}^{t_{k+1}}=\sum_{j=1}^n q_j(t)\big|_{t_k}^{t_{k+1}}, \quad
    \mu^*(Y(t),R_k(t)E)\big|_{t_{k+1}}^{t_{k}}=\sum_{j=1}^n q_j^*(t)\big|_{t_k}^{t_{k+1}},
  \end{equation}
  where we used the facts that the rank of the upper block of $R_k(t)E$ and rank of the
  Wronskian $W(R_k(t)E,Y(t))$ are constant for $t\in[t_k,t_{k+1}]$ according to
  \eqref{E:tY.def.transformed.new}. By the first equation in \eqref{E:osc.number.q.hlp0} we obtain,
  with the telescope summation (and using $t_0=a$ and $t_p=b$), that
  \begin{equation*}
    \Nab{Y} \overset{\eqref{E:Nab.D.R.def}}{=} \sum_{k=0}^{p-1} \mu(Y(t),R_k(t)E)\big|_{t_k}^{t_{k+1}}
      \overset{\eqref{E:osc.number.q.hlp0}}{=}
      \sum_{j=1}^n \sum_{k=1}^{p-1} q_j(t)\big|_{t_k}^{t_{k+1}}
      = \sum_{j=1}^n \big(q_j(b)-q_j(a)\big),
  \end{equation*}
  which completes the proof of equation \eqref{E:osc.number.q}. Analogously, by the second equality in
  \eqref{E:osc.number.q.hlp0} and by \eqref{E:Nab*.D.R.def} we obtain with the telescope summation that
  \begin{equation*}
    \Nsab{Y} \overset{\eqref{E:Nab*.D.R.def}}{=} \sum_{k=0}^{p-1} \mu^*(Y(t),R_k(t)E)\big|_{t_{k+1}}^{t_{k}}
      \overset{\eqref{E:osc.number.q.hlp0}}{=}
      \sum_{j=1}^n \sum_{k=1}^{p-1} q_j^*(t)\big|_{t_k}^{t_{k+1}}
      = \sum_{j=1}^n \big(q_j^*(b)-q_j^*(a)\big),
  \end{equation*}
  which completes the proof of equation \eqref{E:osc.number.q*}.
\end{proof}

\begin{Remark} \label{R:osc.number.independent}
  The results in Theorem~\ref{T:osc.number.q.q*} confirm that the values in \eqref{E:Nab.D.R.def} and
  \eqref{E:Nab*.D.R.def} indeed do not depend on the chosen partition $D=\{t_k\}_{k=0}^p$ of $[a,b]$
  as well as on the chosen system $R=\{R_k(t)\}_{k=0}^{p-1}$ of symplectic matrix-valued functions
  satisfying \eqref{E:tY.def.transformed}.
\end{Remark}

As a~consequence of Theorem~\ref{T:osc.number.q.q*} we obtain the value of the oscillation number and the
dual oscillation number of $Y$, when the upper block $X(t)$ of $Y(t)$ has constant rank on $[a,b]$, see
also \cite[Proposition~3.2(ii)]{jvE2020a} and \cite[Proposition~3.3(iii)]{jvE2020?c}.

\begin{Corollary} \label{C:osc.number.const.rank}
  Assume that $Y$ is a~continuous Lagrangian path on $[a,b]$ such that its upper block $X(t)$, according
  to the notation in \eqref{E:JYH.partition}, has constant rank on $[a,b]$. Then we have
  \begin{equation} \label{E:osc.number.const.rank}
    \Nab{Y}=0, \quad \Nsab{Y}=0.
  \end{equation}
\end{Corollary}
\begin{proof}
  The result follows from Theorem~\ref{T:osc.number.q.q*} and from Lemma~\ref{L:cons.rank} (with
  $\tau_1:=a$ and $\tau_2:=b$).
\end{proof}

The following result presents the invariance of the oscillation number and the dual oscillation number under
a~special continuous symplectic transformation.

\begin{Corollary} \label{C:Low.Triang.trans}
  Let $Y$ be a~continuous Lagrangian path on $[a,b]$ and let $L:[a,b]\to\Rbb^{2n\times2n}$ be a~continuous
  symplectic lower block triangular matrix-valued function, i.e., the matrix $L(t)$ has the form as in
  \eqref{prop2}. Then we have
  \begin{align}
    \Nab{Y} &= \Nab{L\hspt Y}, \label{E:osc.number.L} \\
    \Nsab{Y} &= \Nsab{L\hspt Y}. \label{E:osc.number.L*}
  \end{align}
  \end{Corollary}
\begin{proof}
  The main role in the proof is played by the invariant property of the comparative index with respect
  to block lower triangular symplectic transformations, see \eqref{prop2}. Obviously, $L\hspt Y$ is
  a~continuous Lagrangian path on $[a,b]$. Under the assumptions and the notation of
  Theorem~\ref{T:osc.number.q.q*}, we show that equations \eqref{E:osc.number.q} and
  \eqref{E:osc.number.q*} are also valid for $\Nab{L\hspt Y}$ and $\Nsab{L\hspt Y}$. Applying
  Definition~\ref{D:osc.number} to the path $L\hspt Y$ we choose a~partition $D=\{t_k\}_{k=0}^p$ of
  $[a,b]$ and a~system  $\tR:=\{\tR_k(t)\}_{k=0}^{p-1}$ of symplectic matrix-valued functions
  $\tilde R_k(t)$ such that
  \begin{equation}\label{transvers-LY}
    \rank W(\tilde R_k(t)\hspt E,L(t)\hspt Y(t)) \text{ and } \rank((I\ 0)\hspt\tR_k(t)\hspt E)
    \text{ are constant on } [t_k,t_{k+1}]
  \end{equation}
  for every $k\in\{0,\dots,p-1\}$. By using the assumption that the matrix $L(t)$ is symplectic block lower
  triangular one can rewrite \eqref{transvers-LY} in the equivalent form
  \begin{equation}\label{transvers-LYeq}
    W(L^{-1}(t)\hspt\tilde R_k(t)\hspt E,Y(t)) \text{ and }
    \rank((I\ 0)\hspt L^{-1}(t)\hspt\tilde R_k(t)\hspt E) \text{ are constant on } [t_k,t_{k+1}],
  \end{equation}
  where we used that $\J\nhspt L(t)=L^{T-1}(t)\J$. It follows from \eqref{transvers-LYeq} that one can use
  the same partition for $L(t)\hspt Y(t)$ and $Y(t)$ with the corresponding transformation matrices
  $\tilde R_k(t)$ and $L^{-1}(t)\hspt\tilde R_k(t)$. Finally, by \eqref{prop2} we have, instead of
  \eqref{E:osc.number.q.hlp0}, the equalities
  \begin{align*}
     \mu(L(t)\hspt Y(t),\tilde R_k(t)\hspt E)\big|_{t_k}^{t_{k+1}}
       &\overset{\eqref{prop2}}{=} \mu(Y(t),L^{-1}(t)\hspt \tilde R_k(t)\hspt E)\big|_{t_k}^{t_{k+1}}
       =\sum_{j=1}^n q_j(t)\big|_{t_k}^{t_{k+1}}, \\
     \mu^*(L(t)\hspt Y(t),\tilde R_k(t)\hspt E)\big|_{t_{k+1}}^{t_{k}}
       &\overset{\eqref{prop2}}{=} \mu^*(Y(t),L^{-1}(t)\hspt\tilde R_k(t)\hspt E)\big|_{t_{k+1}}^{t_{k}}
       =\sum_{j=1}^n q_j^*(t)\big|_{t_k}^{t_{k+1}},
  \end{align*}
  where by analogy with the proof of Theorem~\ref{T:osc.number.q.q*} we applied Corollary~\ref{C:const.ranks}
  with the matrix $\hY(t):=L^{-1}(t)\hspt\tilde R_kE$. Summing the equalities derived above for all
  $k\in\{0,\dots,p-1\}$ we complete the proof of \eqref{E:osc.number.L} and \eqref{E:osc.number.L*}.
\end{proof}

Based on Corollary~\ref{C:Low.Triang.trans} we are able to prove the invariance of the
oscillation number and the dual oscillation number in the sense that every continuous
symplectic transformation of a~continuous Lagrangian path can be realized with
a~special continuous symplectic and orthogonal transformation.

\begin{Theorem} \label{T:osc.invariance.transf}
  Let $Y$ be a~continuous Lagrangian path on $[a,b]$ and let $S(t)$ be a~continuous
  symplectic matrix on $[a,b]$. Then we have the equalities
  \begin{align}
    \Nab{S^{-1}\hspt Y} &= \Nab{Z_{SE}^{-1}\hspt\hspt Y},
      \label{E:osc.invariance.transf} \\
    \Nsab{S^{-1}\hspt Y} &= \Nsab{Z_{SE}^{-1}\hspt\hspt Y},
      \label{E:osc.invariance.transf*}
  \end{align}
  where the symplectic and orthogonal matrix $Z_{SE}(t)=Z_{S(t)E}$ is
  defined according to \eqref{E:ZY.def}.
\end{Theorem}
\begin{proof}
  By property \eqref{E:SZYL} in Remark~\ref{R:S=ZL} we have $S(t)=Z_{S(t)E}\hspt L(t)$, where $L(t)$ is
  a~continuous symplectic block lower triangular matrix. Then in view of Corollary~\ref{C:Low.Triang.trans}
  we have
  \begin{align*}
    \Nab{S^{-1}\hspt Y} &= \Nab{L^{-1}Z_{SE}^{-1}\hspt\hspt Y} \overset{\eqref{E:osc.number.L}}{=}
      \Nab{Z_{SE}^{-1}\hspt\hspt Y}, \\
    \Nsab{S^{-1}\hspt Y} &= \Nsab{L^{-1}Z_{SE}^{-1}\hspt\hspt Y} \overset{\eqref{E:osc.number.L*}}{=}
      \Nsab{Z_{SE}^{-1}\hspt\hspt Y},
  \end{align*}
  which completes the proof.
\end{proof}

Next we present a~formula relating the oscillation number $\Nab{Y}$ and the dual
oscillation number $\Nsab{Y}$, see also \cite[Proposition~3.3(iv)]{jvE2020?c} and
\cite[Theorem~5.1]{pS.rSH2017}.

\begin{Theorem} \label{T:osc.dual.osc.number}
  Let $Y$ be a~continuous Lagrangian path on $[a,b]$. Then the oscillation number and
  the dual oscillation number of\/ $Y$ on $[a,b]$ are related by the formula
  \begin{equation} \label{E:osc.dual.osc.number}
    \Nab{Y}+\rank X(b)=\Nsab{Y}+\rank X(a),
  \end{equation}
  where $X(t)$ is the upper block of $Y(t)$ as in \eqref{E:JYH.partition}.
\end{Theorem}
\begin{proof}
  The result follows from Theorem~\ref{T:osc.number.q.q*} and from the relationship
  between the integers $q_j(t)$ and $q_j^*(t)$ in \eqref{E:qj.qj*}. Namely, we have
  \begin{align*}
    \Nab{Y}+\rank X(b) &\overset{\eqref{E:osc.number.q}}{=}
      \rank X(b)+\sum_{j=1}^n\big(q_j(b)-q_j(a)\big) \\
    &\overset{\eqref{E:qj.qj*}}{=}
      \rank X(b)+\defect X(b)-\defect X(a)+\sum_{j=1}^n\big(q_j^*(b)-q_j^*(a)\big) \\
    &\hspace*{-0.6mm}\overset{\eqref{E:osc.number.q*}}{=} \Nsab{Y}+\rank X(a),
  \end{align*}
  which proves the result in \eqref{E:osc.dual.osc.number}.
\end{proof}

\begin{Corollary} \label{C:compare.generalized}
  Let $S(t)$ be a~continuous symplectic matrix on $[a,b]$, which is partitioned into
  $n\times n$ blocks as in \eqref{E:S.def}. Then we have the equalities
  \begin{align}
    \Nab{S\nhspt E}+\Nab{S^{-1}\nhspt E} &= \rank S_{12}(a)-\rank S_{12}(b),
      \label{E:compare.generalized} \\
    \Nsab{S\nhspt E}+\Nsab{S^{-1}\nhspt E} &= \rank S_{12}(b)-\rank S_{12}(a).
      \label{E:compare.generalized*}
  \end{align}
\end{Corollary}
\begin{proof}
  According to \eqref{f1}, for the continuous Lagrangian path $Y(t):=S(t)E$ on $[a,b]$ we have
  \begin{equation}\label{E:compare.generalized.Ns.N}
    \Nab{S\nhspt E}=-\Nsab{S^{-1}\nhspt E}, \quad \Nsab{S\nhspt E}=-\Nab{S^{-1}\nhspt E}.
  \end{equation}
  Equations \eqref{E:compare.generalized} and \eqref{E:compare.generalized*} now follow directly
  from formula \eqref{E:osc.dual.osc.number} (with $Y(t):=S(t)\hspt E$ and $X(t):=S_{12}(t)$)
  and from the relations in \eqref{E:compare.generalized.Ns.N}. The proof is complete.
\end{proof}

In the final part of this section we will discuss some additional properties of the oscillation number
and the dual oscillation number, which are based on Theorem~\ref{T:osc.number.q.q*}. At the first place
we obtain the following representations of the changes in the corresponding special arguments $\ArgL$ and
$\ArgR$ of the matrix $Z_Y(t)$.

\begin{Corollary} \label{C:Arg.Lidskii.osc}
  Let $Y$ be a~continuous Lagrangian path on $[a,b]$ and let $Z_Y(t)$ be the
  continuous symplectic and orthogonal matrix defined in \eqref{E:ZY.def}. Then
  we have the representations
  \begin{align}
    \ArgL(Z_Y(b))-\ArgL(Z_Y(a)) &= \Arg_3(Z_Y(b))-\Arg_3(Z_Y(a))-\pi\,\Nab{Y},
      \label{E:Arg.Lidskii.osc} \\
    \ArgR(Z_Y(b))-\ArgR(Z_Y(a)) &= \Arg_3(Z_Y(b))-\Arg_3(Z_Y(a))-\pi\,\Nsab{Y}.
      \label{E:Arg.Lidskii.osc*}
  \end{align}
\end{Corollary}
\begin{proof}
  The result in \eqref{E:Arg.Lidskii.osc} follows by combining formula
  \eqref{EJ:Arg3.ArgL} at $t=a$ and $t=b$ with \eqref{E:osc.number.q} in
  Theorem~\ref{T:osc.number.q.q*}, while the result in \eqref{E:Arg.Lidskii.osc*}
  follows by combining formula \eqref{EJ:Arg3.ArgR} at $t=a$ and $t=b$ with
  \eqref{E:osc.number.q*} in Theorem~\ref{T:osc.number.q.q*}.
\end{proof}

The representations of the oscillation number and the dual oscillation number in
Theorem~\ref{T:osc.number.q.q*} yield their additive property with respect to a~block
diagonal structure of the components $X$ and $U$ of the Lagrangian path $Y$. More
precisely, consider the dimensions $n_1,n_2\in\Nbb$ and the permutation matrix
\begin{equation} \label{E:Pi.def}
  \Pi:=\mmatrix{I_{n_1} & 0 & 0 & 0 \\ 0 & 0 & I_{n_2} & 0 \\ 0 & I_{n_1} & 0 & 0 \\
    0 & 0 & 0 & I_{n_2}}
\end{equation}
of dimension $2(n_1+n_2)\times2(n_1+n_2)$. Here $I_k$ denotes the $k\times k$ identity
matrix. The following result will be useful in particular for the construction of
higher dimensional examples. We recall the convention that $\diag\{A,B\}$ denotes the
block diagonal matrix with the matrices $A$ and $B$ on the diagonal.

\begin{Theorem} \label{T:diag.form.osc}
  Assume that $Y_1=(X_1^T,U_1^T)^T$ and $Y_2=(X_2^T,U_2^T)^T$ are continuous
  Lagrangian paths on $[a,b]$ with values in $\Rbb^{2n_1\times n_1}$
  and \/$\Rbb^{2n_2\times n_2}$, respectively. Consider the continuous Lagrangian
  path on $[a,b]$ defined by
  \begin{equation}  \label{E:diag.Y.def}
    Y:=\Pi\,\diag\{Y_1,\hspt Y_2\}
    =\mmatrix{\diag\{X_1,X_2\} \\[0.5mm] \diag\{U_1,U_2\}}
  \end{equation}
  with values in $\Rbb^{2n\times n}$, where $n:=n_1+n_2$ and where the matrix $\Pi$ is
  given by \eqref{E:Pi.def}. Then
  \begin{align}
    \Nab{Y} &= \Nab{Y_1}+\Nab{Y_2}, \label{E:diag.osc} \\
    \Nsab{Y} &= \Nsab{Y_1}+\Nsab{Y_2}. \label{E:diag.osc*}
  \end{align}
\end{Theorem}
\begin{proof}
  According to \eqref{E:ZY.def} applied to the Lagrangian paths $Y_1$, $Y_2$, and to
  the Lagrangian path $Y$ defined in \eqref{E:diag.Y.def}, for the corresponding
  symplectic and orthogonal matrices $S_1(t):=Z_{Y_1}(t)$, $S_2(t):=Z_{Y_2}(t)$, and
  $S(t):=Z_Y(t)$ we have
  \begin{equation*}
    S(t)=\Pi\,\diag\{S_1(t),\,S_2(t)\}\,\Pi, \quad t\in[a,b].
  \end{equation*}
  This implies that the matrices $W_{\nhspt S_1}(t)$, $W_{\nhspt S_2}(t)$, and
  $W_{\nhspt S}(t)$ defined through equation \eqref{E:W.def} satisfy
  \begin{equation} \label{E:diag.form.osc.hlp2}
    W_{\nhspt S}(t)=\diag\{W_{\nhspt S_1}(t),\hspt W_{\nhspt S_2}(t)\}, \quad t\in[a,b].
  \end{equation}
  Formula \eqref{E:diag.form.osc.hlp2} shows that the Lidskii angles
  $\varphi_j(t)$ for $j\in\{1,\dots,n\}$ of the Lagrangian path $Y$ according
  to Definition~\ref{D:Lidskii.Lagrangian.path} consists exactly of the Lidskii angles
  $\varphi_j^{[1]}(t)$ for $j\in\{1,\dots,n_1\}$ and $\varphi_j^{[2]}(t)$ for
  $j\in\{1,\dots,n_2\}$ of the Lagrangian paths $Y_1$ and $Y_2$. Therefore,
  equalities \eqref{E:diag.osc} and \eqref{E:diag.osc*} follow respectively from
  formula \eqref{E:osc.number.q} for the oscillation numbers $\Nab{Y}$, $\Nab{Y_1}$,
  $\Nab{Y_2}$ and from formula \eqref{E:osc.number.q*} for the dual oscillation
  numbers $\Nsab{Y}$, $\Nsab{Y_1}$, $\Nsab{Y_2}$.
\end{proof}

\section{Oscillation numbers and Maslov index} \label{S:Maslov}
In this section we make a~connection of the oscillation number and the dual oscillation
number of two continuous Lagrangian paths $Y$ and $\hY$ on $[a,b]$ with the Maslov index.
Here we use the definition of the Maslov index $\Masab{Y}{\hY}$ from
\cite[Definition~1.5]{bBB.kF1998}, which we recall.

Let $Y$ and $\hY$ be continuous Lagrangian paths on $[a,b]$ with their partitions into
$n\times n$ blocks as in \eqref{E:JYH.partition}. According to the latter reference,
see also \cite[Section~1]{pH.yL.aS2017}, we consider the complex $n\times n$ matrix
\begin{equation} \label{E:GGamma.def}
  \GGamma(t):= -[X(t)+i\hspt U(t)]\,[X(t)-i\hspt U(t)]^{-1}\,
    [\hX(t)-i\hspt\hU(t)]\,[\hX(t)+i\hspt\hU(t)]^{-1}, \quad t\in[a,b].
\end{equation}
Then the matrix $\GGamma(t)$ is well-defined, continuous, and unitary on $[a,b]$. The
last property also follows from the proof of Lemma~\ref{L:GGamma.WS} below, see
equation \eqref{E:GGamma.WS.hlp2}. Let $\ggamma_j(t)$
for $j\in\{1,\dots,n\}$ be the eigenvalues of the matrix $\GGamma(t)$, which are
continuous on $[a,b]$ and lie on the unit circle $\Ubb$ in the complex plane. Let us
fix a~point $\tau_0\in[a,b]$. Since the matrix $\GGamma(\tau_0)$ has $n$ (i.e., finitely
many) eigenvalues $\ggamma_j(\tau_0)$, there exists a~neighborhood $\O(\tau_0)$ of
$\tau_0$ and a~number $\eps\in(0,\pi)$ such that the matrix
$\GGamma(t)-\exp\hspt\{i\hspt(\pi\pm\eps)\}\,I$ is invertible for all $t\in\O(\tau_0)$.
This implies by the compactness of $[a,b]$ that there exists a~finite partition
$D=\{a=t_0<t_1\dots<t_p=b\}$ of $[a,b]$ along with numbers $\eps_k\in(0,\pi)$ such that
for every $k\in\{1,\dots,p\}$ the matrix
$\GGamma(t)-\exp\hspt\{i\hspt(\pi\pm\eps_k)\}\,I$ is invertible for all
$t\in(t_{k-1},t_k)$, i.e., $\exp\hspt\{i\hspt(\pi\pm\eps_k)\}$ is not an~eigenvalue
of $\GGamma(t)$ for $t\in(t_{k-1},t_k)$. Moreover, for each $t\in[t_{k-1},t_k]$ there
are at most $n$ angles $\theta\in[0,\eps_k]\subseteq[0,\pi)$ such that
$\exp\hspt\{i\hspt(\pi+\theta)\}$ is an~eigenvalue of $\GGamma(t)$.
This allows to define the number
\begin{equation} \label{E:l.def}
  \ell(t,\eps_k):=\sum_{\theta\in[0,\eps_k)} \!\! \defect \nhspt \Big(
    \GGamma(t)-\exp\hspt\{i\hspt(\pi+\theta)\}\,I \Big), \quad t\in[t_{k-1},t_k].
\end{equation}
Hence, $\ell(t,\eps_k)$ is equal to the number of the eigenvalues of $\GGamma(t)$, which
lie on the arc
\begin{equation*}
  A_k:=\{\exp\hspt(i\hspt\theta),\ \theta\in[\pi,\pi+\eps_k)\}
\end{equation*}
of the unit circle $\Ubb$. By \cite[Definition~1.5]{bBB.kF1998} or
\cite[Definition~1.4]{pH.yL.aS2017} the {\em Maslov index\/} of the continuous
Lagrangian paths $Y$ and $\hY$ is defined as the integer number
\begin{equation} \label{E:Maslov.index.def}
  \Masab{Y}{\hY}:=\sum_{k=1}^p\Big(\ell(t_k,\eps_k)-\ell(t_{k-1},\eps_k)\Big).
\end{equation}
Note that this definition does not depend on the choice of the partition
$D=\{t_k\}_{k=0}^p$ of $[a,b]$ and on the choice of the numbers $\eps_k$, as long as
they satisfy the above properties, see e.g. \cite[pg.~10]{bBB.kF1998}. This property
will also follow from our main result below (see
Theorem~\ref{T:osc.Maslov} and Remark~\ref{R:osc.Maslov1}).

\begin{Lemma} \label{L:GGamma.WS}
  Let $Y$ and\/ $\hY$ be continuous Lagrangian paths on $[a,b]$ and define the matrix
  $\GGamma(t)$ by \eqref{E:GGamma.def}. Consider the symplectic matrix
  \begin{equation} \label{E:St.ZY.ZhY.def}
    S(t):=Z_Y^{-1}(t)\,Z_{\hY}(t)=Z_Y^T\nhspt(t)\,Z_{\hY}(t), \quad t\in[a,b],
  \end{equation}
  where $Z_Y(t)$ and $Z_{\hY}(t)$ are the symplectic and orthogonal matrices associated
  with $Y(t)$ and $\hY(t)$ through \eqref{E:ZY.def}. Moreover, consider the unitary
  and symmetric matrix $W_{\nhspt S}(t):=W_{\nhspt S(t)}$ on $[a,b]$ defined in
  \eqref{E:W.def} which is associated with the above matrix $S(t)$, i.e.,
  \begin{equation} \label{E:GGamma.WS.hlp0}
    W_{\nhspt S}(t)=
      K_{\hY}^{-1}(t)\,[\hspt Y^T\nhspt(t)\hspt\hY(t)+i\hspt W(Y(t),\hY(t))]^{-1}\,
      [\hspt Y^T\nhspt(t)\hspt\hY(t)-i\hspt W(Y(t),\hY(t))]\,K_{\hY}(t).
  \end{equation}
  Then for each $t\in[a,b]$ the matrices $W_{\nhspt S}(t)$ and $-\GGamma(t)$ are similar.
\end{Lemma}
\begin{proof}
  Let us fix any $t\in[a,b]$. For brevity we will suppress the argument $t$ in the
  following calculations. Consider the auxiliary Lagrangian paths
  \begin{equation} \label{E:GGamma.WS.hlp1}
    \Ystar:=Y\nhspt K_Y=\mmatrix{\Xstar \\ \Ustar}, \quad
    \hYstar:=\hY\nhspt K_{\hY}=\mmatrix{\hXstar \\ \hUstar},
  \end{equation}
  which form the second block columns of the matrices $Z_Y$ and $Z_{\hY}$, i.e.,
  $\Ystar=Z_YE$ and $\hYstar=Z_{\hY}E$. Since $\Xstar=X K_Y$, $\Ustar=U K_Y$ and
  $\hXstar=\hX K_{\hY}$, $\hUstar=\hU K_{\hY}$, where the matrices $K_Y$ and $K_{\hY}$
  are invertible, it follows from the definition of $\GGamma$ in \eqref{E:GGamma.def}
  that
  \begin{equation} \label{E:GGamma.WS.hlp2}
    \GGamma=-(\Xstar+i\hspt\Ustar)\,(\Xstar-i\hspt\Ustar)^{-1}\,
      (\hXstar-i\hspt\hUstar)\,(\hXstar+i\hspt\hUstar)^{-1}.
  \end{equation}
  Note that the orthogonality of the matrices $Z_Y$ and $Z_{\hY}$ yields that
  \begin{equation} \label{E:GGamma.WS.hlp3}
    (\Xstar\pm i\hspt\Ustar)^{-1}=\Xstar^T\mp i\hspt\Ustar^T, \quad
    (\hXstar\pm i\hspt\hUstar)^{-1}=\hXstar^T\mp i\hspt\hUstar^T.
  \end{equation}
  We express the matrix $S$ in the form
  \begin{equation} \label{E:GGamma.WS.hlp4}
   S \overset{\eqref{E:St.ZY.ZhY.def}}{=}
     \mmatrix{K_Y\hspt Y^T\hY K_{\hY} & -K_Y\hspt W(Y,\hY)\hspt K_{\hY} \\[1mm]
       K_Y\hspt W(Y,\hY)\hspt K_{\hY} & K_Y\hspt Y^T\hY K_{\hY}}
     \overset{\eqref{E:GGamma.WS.hlp1}}{=}
       \mmatrix{\Ystar^T\hYstar & -W(\Ystar,\hYstar) \\[1mm]
         W(\Ystar,\hYstar) & \Ystar^T\hYstar}.
  \end{equation}
  Then according to the form of the matrix $W_{\nhspt S}$ in \eqref{E:GGamma.WS.hlp0}
  we have
  \begin{equation} \label{E:GGamma.WS.hlp5}
    W_{\nhspt S} \overset{\eqref{E:GGamma.WS.hlp4}}{=}
      [\hspt\Ystar^T\hYstar+i\hspt W(\Ystar,\hYstar)]^{-1}\,
      [\hspt\Ystar^T\hYstar-i\hspt W(\Ystar,\hYstar)]=L^{T-1}\bar L^T,
  \end{equation}
  where the $n\times n$ matrix $L$ is defined by
  \begin{equation} \label{E:GGamma.WS.hlp6}
    L:=\hYstar^T\Ystar-i\hspt W(\hYstar,\Ystar)
      =[\hspt\Ystar^T\hYstar+i\hspt W(\Ystar,\hYstar)]^T, \quad
    L^{-1}=\bar L^T.
  \end{equation}
  This means that $L$ is a~unitary matrix. Then we obtain that
  \begin{align*}
    -\GGamma &\overset{\eqref{E:GGamma.WS.hlp2}}{=}
      (\Xstar+i\hspt\Ustar)\,(\Xstar-i\hspt\Ustar)^{-1}\,
      (\hXstar-i\hspt\hUstar)\,(\hXstar+i\hspt\hUstar)^{-1} \\
    &\overset{\eqref{E:GGamma.WS.hlp3}}{=}
      (\Xstar+i\hspt\Ustar)\,
      \underbrace{(\Xstar^T+i\hspt\Ustar^T)\,(\hXstar-i\hspt\hUstar)}\,
      \underbrace{(\hXstar^T-i\hspt\hUstar^T)\,(\Xstar+i\hspt\Ustar)}\,
      (\Xstar+i\hspt\Ustar)^{-1} \\
    &\hspace*{1.3mm}= (\Xstar+i\hspt\Ustar)\,
      [\hspt\Ystar^T\hYstar-i\hspt W(\Ystar,\hYstar)]\,
      [\hspt\hYstar^T\Ystar+i\hspt W(\hYstar,\Ystar)]
      (\Xstar+i\hspt\Ustar)^{-1} \\
    &\hspace*{-1mm}\overset{\eqref{E:GGamma.WS.hlp6}}{=}
      (\Xstar+i\hspt\Ustar)\,L^{-1}\bar L\,(\Xstar+i\hspt\Ustar)^{-1}
      \overset{\eqref{E:GGamma.WS.hlp5}}{=}
      (\Xstar+i\hspt\Ustar)\,L^{-1}\,W_{\nhspt S}\,L\,(\Xstar+i\hspt\Ustar)^{-1},
  \end{align*}
  where in the last step we used the symmetry of the matrix $W_{\nhspt S}$. Therefore,
  we showed that the matrices $-\GGamma$ and $W_{\nhspt S}$ are similar, which
  completes the proof of this lemma.
\end{proof}

Based on the above preliminary considerations we can now prove the main result of this
section, which connects the Maslov index with the Lidskii angles and hence with the
oscillation number through equation \eqref{E:Maslov.osc.number}. More precisely, the
Maslov index of $Y$ and $\hY$ over the interval $[a,b]$ can be calculated as the
total change in the interval $[a,b]$ of the integers $q_j(t)$, which are associated
through \eqref{E:qjt.def} or \eqref{E:qj.qj*.floor.ceiling} with the continuous Lidskii
angles $\varphi_j(t)$ of the symplectic matrix $S(t)$ defined in \eqref{E:St.ZY.ZhY.def},
compare with \cite[Definition~2.2]{yZ.lW.cZ2018} and \cite[Eqs.~(2.4)--(2.5)]{bBB.cZ2018}.
Consequently, the Maslov index is equal to the oscillation number of the transformed
Lagrangian path $Z_Y^{-1}\hspt\hY$ on $[a,b]$.

\begin{Theorem} \label{T:osc.Maslov}
  Let $Y$ and\/ $\hY$ be continuous Lagrangian paths on $[a,b]$ and let the symplectic
  and orthogonal matrices $Z_Y(t)$, $Z_{\hY}(t)$ together with the invertible matrices
  $K_Y(t)$, $K_{\hY}(t)$ be defined according to
  \eqref{E:ZY.def} on $[a,b]$. Then
  \begin{equation} \label{E:osc.Maslov.Lidskii}
    \Masab{Y}{\hY}=\sum_{j=1}^n \big( \hspt q_j(b)-q_j(a) \big),
  \end{equation}
  where $q_j(t)$ are the integers associated through \eqref{E:qjt.def} with the continuous
  Lidskii angles $\varphi_j(t)$ of the symplectic matrix $S(t)$ defined in \eqref{E:St.ZY.ZhY.def}.
  Consequently, for any continuous symplectic matrix $Z(t)$ such that $Y(t)=Z(t)\hspt E$ on $[a,b]$
  we have
  \begin{equation} \label{E:osc.Maslov}
    \Masab{Y}{\hY}=\Nab{Z^{-1}\hY}=\Nab{\tY}, \quad
    \tY:=Z_{Y}^{-1}\hspt\hY=\mmatrix{-K_Y\hspt W(Y,\hY) \\ K_Y\hspt Y^T\hY}.
  \end{equation}
\end{Theorem}
\begin{proof}
  Let the matrices $Z_Y(t)$, $Z_{\hY}(t)$ and $K_Y(t)$, $K_{\hY}(t)$ be as in the
  theorem. Since the matrix $Z_Y(t)$ is orthogonal, the form of the transformed
  Lagrangian path $\tY$ defined in \eqref{E:osc.Maslov} follows from the same
  calculation as in \eqref{E:GGamma.WS.hlp4}. Consider the matrix $\GGamma(t)$ defined
  in \eqref{E:GGamma.def} on $[a,b]$ and a~partition $D=\{t_k\}_{k=0}^p$ of $[a,b]$ with
  the numbers $\eps_k\in(0,\pi)$, which are used in \eqref{E:l.def} in order to define
  the numbers $\ell(t,\eps_k)$. With the symplectic matrix $S(t)$ in
  \eqref{E:St.ZY.ZhY.def} we consider the corresponding matrix $W_{\nhspt S}(t)$
  in \eqref{E:GGamma.WS.hlp0} on $[a,b]$. Denote by $\varphi_j(t)$ for
  $j\in\{1,\dots,n\}$ the continuous Lidskii angles of the symplectic matrix $S(t)$
  on $[a,b]$ with the corresponding integers $q_j(t)$ satisfying \eqref{E:qjt.def}.
  Since by Lemma~\ref{L:GGamma.WS} the matrices $\GGamma(t)$ and $-W_{\nhspt S}(t)$
  are similar for all $t\in[a,b]$, then these matrices have the same eigenvalues. And
  since the arguments of the eigenvalues of the matrices $-W_{\nhspt S}(t)$ and
  $W_{\nhspt S}(t)$ differ by $\pi$, we obtain from \eqref{E:l.def} that
  \begin{equation} \label{E:osc.Maslov.hlp2}
    \ell(t,\eps_k)=\sum_{\theta\in[0,\eps_k)} \!\! \defect \nhspt \Big(
      W_{\nhspt S}(t)-\exp\hspt(i\hspt\theta)\,I \Big), \quad t\in[t_{k-1},t_k].
  \end{equation}
  Equation \eqref{E:osc.Maslov.hlp2} shows that $\ell(t,\eps_k)$ is equal to the
  number of the eigenvalues of the matrix $W_{\nhspt S}(t)$, which lie on the arc
  $B_k:=\{\exp\hspt(i\hspt\theta),\ \theta\in[0,\eps_k)\}$ of the unit circle $\Ubb$.
  Hence, the changes (i.e., incrementing or decrementing) of the integers
  $\ell(t,\eps_k)$ when the eigenvalues of $\GGamma(t)$ pass through $-1$, as it is
  commented in \cite[pp.~796--797]{pH.yL.aS2017}, can be calculated by the same
  changes when the eigenvalues of $W_{\nhspt S}(t)$ pass through $1$. More precisely,
  the number $\ell(t,\eps_k)$ increases by one if and only if there is a~corresponding
  Lidskii angle $\varphi_j(t)$ of the matrix $S(t)$, which arrives from below
  to an~integer multiple of $2\pi$. And the number $\ell(t,\eps_k)$ decreases by one
  if and only if there is a~corresponding Lidskii angle $\varphi_j(t)$ of $S(t)$,
  which leaves an~integer multiple of $2\pi$ in the downward direction. Consequently,
  the sum of the numbers $q_j(t)$ increases or decreases in the interval $[t_{k-1},t_k]$
  by the same amount as $\ell(t,\eps_k)$, and thus
  \begin{equation} \label{E:osc.Maslov.hlp4}
    \ell(t_k,\eps_k)-\ell(t_{k-1},\eps_k)=\sum_{j=1}^n \big(
    \hspt q_j(t_k)-q_j(t_{k-1})\big).
  \end{equation}
  By the definition of the Maslov index of $Y$ and $\hY$ on $[a,b]$ in
  \eqref{E:Maslov.index.def} we then obtain
  \begin{equation} \label{E:osc.Maslov.hlp5}
    \Masab{Y}{\hY} \overset{\eqref{E:osc.Maslov.hlp4}}{=}
      \sum_{k=1}^p \sum_{j=1}^n \big( \hspt q_j(t_k)-q_j(t_{k-1}) \big)
      =\sum_{j=1}^n \big( \hspt q_j(b)-q_j(a) \big),
  \end{equation}
  which shows the result in \eqref{E:osc.Maslov.Lidskii}. By combining equality
  \eqref{E:osc.Maslov.hlp5} and \eqref{E:osc.number.q} in
  Theorem~\ref{T:osc.number.q.q*} (with the continuous Lagrangian path
  $Y:=S\nhspt E=\tY\nhspt K_{\hY}$) we conclude that
  \begin{equation*}
    \Masab{Y}{\hY} \overset{\eqref{E:osc.Maslov.hlp5}}{=}
      \sum_{j=1}^n \big( \hspt q_j(b)-q_j(a) \big)
      \overset{\eqref{E:osc.number.q}}{=} \Nab{\tY\nhspt K_{\hY}}
      \overset{\eqref{E:osc.number.multiple}}{=} \Nab{\tY},
  \end{equation*}
  where in the last step we used the invariance of the oscillation number with respect
  to the multiplication of $\tY$ by a~continuous invertible $n\times n$ matrix function
  from the right (see Remark~\ref{R:osc.number.multiple}(i)). Therefore, the second
  equality in \eqref{E:osc.Maslov} is proved. Next we apply Theorem~\ref{T:osc.invariance.transf}
  with $S(t):=Z(t)$ to prove the first equality in \eqref{E:osc.Maslov} for an~arbitrary continuous
  symplectic matrix $Z(t)$ such that $Z(t)\hspt E=Y(t)$. The proof is complete.
\end{proof}

\begin{Remark} \label{R:osc.Maslov1}
  \par(i) Equation \eqref{E:osc.Maslov.Lidskii} also proves the independence of the value
  $\Masab{Y}{\hY}$ defined in \eqref{E:Maslov.index.def} on the choice of the therein
  used partition $D=\{t_k\}_{k=0}^p$ and numbers $\eps_k$.
  \par(ii) According to the definition in \eqref{E:GGamma.def} the Maslov index is invariant under
  the multiplication of its arguments $Y$ and $\hY$ by invertible $n\times n$ matrices from the
  right, i.e.,
  \begin{equation*}
    \Masab{YC}{\hY\hat C}=\Masab{Y}{\hY}
  \end{equation*}
  for arbitrary continuous nonsingular matrices $C(t)$ and $\hat C(t)$ on $[a,b]$. This fact also
  follows from formula \eqref{E:osc.Maslov} and from the invariant properties of the oscillation
  numbers according to Remark~\ref{R:osc.number.multiple}(i) and Corollary~\ref{C:Low.Triang.trans}.
  In particular, the symplectic matrix $Z(t)$ in Theorem~\ref{T:osc.Maslov} can be chosen to be in
  a~more general form $Z(t)\hspt E=Y(t)\hspt C(t)$ on $[a,b]$.
\end{Remark}

The result in Theorem~\ref{T:osc.Maslov} leads to the evaluation of the oscillation number in
terms of the Maslov index, as we announced in equation \eqref{E:osc.number.Maslov}.

\begin{Corollary} \label{C:osc.number.Maslov}
  Let $Y$ be a~continuous Lagrangian path on $[a,b]$. Then
  \begin{equation} \label{E:osc.number.Maslov.C}
    \Nab{Y}=\Masab{E}{Y},
  \end{equation}
  where the matrix $E$ defined in \eqref{E:E.def} represents the constant vertical Lagrangian path.
\end{Corollary}
\begin{proof}
  Formula \eqref{E:osc.number.Maslov.C} follows from equation \eqref{E:osc.Maslov} in
  Theorem~\ref{T:osc.Maslov} (with $Y:=E$ and $\hY:=Y$), since in this case the matrix
  $Z_Y(t)=Z_E(t)\equiv I$ on $[a,b]$.
\end{proof}

Next we discuss a~dual notion to the Maslov index $\Masab{Y}{\hY}$, which reflects the results
regarding the dual oscillation numbers obtained in Sections~\ref{S:Osc.number} and~\ref{S:Osc.compare}.

\begin{Remark} \label{R:dual.Maslov.index}
  Let $Y$ and $\hY$ be two given continuous Lagrangian paths on $[a,b]$. Assume that
  $D=\{a=t_0<t_1\dots<t_p=b\}$ is a~partition of $[a,b]$, which is used along
  with numbers $\eps_k\in(0,\pi)$ in the definition of the Maslov index
  $\Masab{Y}{\hY}$ in \eqref{E:Maslov.index.def}. Then instead of the numbers
  $\ell(t,\eps_k)$ defined in \eqref{E:l.def} we consider the numbers
  \begin{equation} \label{E:l.def*}
    \ell^{\hspt*}(t,\eps_k):=\sum_{\theta\in[0,\eps_k)} \!\! \defect \nhspt \Big(
    \GGamma(t)-\exp\hspt\{i\hspt(\pi-\theta)\}\,I \Big), \quad t\in[t_{k-1},t_k].
  \end{equation}
  This means that $\ell^{\hspt*}(t,\eps_k)$ is equal to the number of the eigenvalues
  of $\GGamma(t)$, which lie on the arc
  $A_k^*:=\{\exp\hspt(i\hspt\theta),\ \theta\in(\pi-\eps_k,\pi]\}$ of the unit circle
  $\Ubb$. By analogy with \eqref{E:Maslov.index.def} we define the
  {\em dual Maslov index\/} of the continuous Lagrangian paths $Y$ and $\hY$ as the
  integer
  \begin{equation} \label{E:Maslov.index.def*}
    \Massab{Y}{\hY}:=\sum_{k=1}^p \Big(
      \ell^{\hspt*}(t_{k-1},\eps_k)-\ell^{\hspt*}(t_k,\eps_k) \Big).
  \end{equation}
  Then similarly as in Theorem~\ref{T:osc.Maslov} we obtain that for an~arbitrary continuous
  symplectic matrix $Z(t)$ associated with $Y(t)$ via the condition $Y(t)=Z(t)\hspt E$
  \begin{equation} \label{E:osc.Maslov*}
    \Massab{Y}{\hY}=\Nsab{Z^{-1}\hY}=\Nsab{Z_Y^{-1}\hY}.
      \end{equation}
  In view of \eqref{f1}, the representation formulas \eqref{E:osc.Maslov} and \eqref{E:osc.Maslov*}
  for the Maslov index and the dual Maslov index imply that
  \begin{equation} \label{E:osc.Maslov*.Maslov1}
    \Massab{Y}{\hY}=\Nsab{Z^{-1}\hY}=-\Nab{\hZ^{-1}Y}=-\Masab{\hY}{Y},
  \end{equation}
  where the continuous symplectic matrix $\hZ(t)$ is such that $\hY(t)=\hZ(t)\hspt E$.
  This yields, in view of Theorem~\ref{T:osc.number.q.q*}, the geometric
  interpretation of the dual Maslov index $\Massab{Y}{\hY}$ as the total change in
  the interval $[a,b]$ of the integers $q_j^*(t)$, which are associated through
  \eqref{E:qjt*.def} or \eqref{E:qj.qj*.floor.ceiling} with the continuous Lidskii
  angles $\varphi_j(t)$ of the symplectic matrix $S(t)$ defined in
  \eqref{E:St.ZY.ZhY.def}. From \eqref{E:osc.Maslov*} we then obtain the dual version
  of Corollary~\ref{C:osc.number.Maslov} in the form
  \begin{equation} \label{E:osc.number.Maslov.C*}
    \Nsab{Y}=\Massab{E}{Y}.
  \end{equation}

  Finally, from Theorem~\ref{T:osc.dual.osc.number} we obtain the equality
  \begin{equation} \label{E:osc.Maslov*.Maslov2}
    \Massab{Y}{\hY}=\Masab{Y}{\hY}+\rank W(Y(b),\hY(b))-\rank W(Y(a),\hY(a)).
  \end{equation}
  In view of formulas \eqref{E:osc.number.Maslov.C} and \eqref{E:osc.number.Maslov.C*},
  the results in \eqref{E:osc.number.const.rank} and \eqref{E:osc.Maslov*.Maslov1},
  \eqref{E:osc.Maslov*.Maslov2} correspond to the vanishing property and to the
  flipping property (or the symmetry) of the Maslov index and the dual Maslov index,
  see \cite[Proposition~2.3.1\,(e),~(f)]{bBB.cZ2018} and
  \cite[Property~XI, pg.~130]{seC.rL.eyM1994}.
\end{Remark}

In \cite[Theorem~2.3]{jvE2019}, \cite[Propositions~3.2 and~3.4]{jvE2020a}, and
\cite[Proposition~3.3(v)]{jvE2020?c} we proved that the oscillation number and the dual
oscillation number count the left and right proper focal points of a~conjoined basis $Y$
of system \eqref{H} satisfying the Legendre condition \eqref{E:LC}. Below we extend these
properties to the case when $Y$ is an~arbitrary piecewise continuously differentiable
Lagrangian path on $[a,b]$. For this purpose we recall that if $A(t)$ is a~given matrix defined
on the interval $[a,b]$, then $\rank A(t_0^\pm)$ denote the left-hand and the right-hand limits
of the quantity $\rank A(t)$ at the point $t_0\in[a,b]$.

\begin{Theorem} \label{T:monoton.osc}
  Let $Y$ be a~piecewise continuously differentiable Lagrangian path $[a,b]$ with the partition
  as in \eqref{E:JYH.partition} and assume that
  \begin{equation}\label{derWr}
    [Y'(t)]^T\!\J\hspt Y(t)=-Y^T\nhspt(t)\hspt\J\hspt Y'(t) \geq 0, \quad t\in[a,b].
  \end{equation}
  Then the subspace $\Im X(t)$ is piecewise constant on $[a,b]$ and
  \begin{align}
    \Nab{Y} &= \sum_{t_0\in(a,b]}\!\!\big(\nhspt\rank X(t_0^-)-\rank X(t_0)\big)\geq0, \label{Nge0} \\
    \Nsab{Y} &= \sum_{t_0\in[a,b)}\!\!\big(\nhspt\rank X(t_0^+)-\rank X(t_0)\big)\geq0, \label{N*ge0}
  \end{align}
  where the sums in equations \eqref{Nge0} and \eqref{N*ge0} are finite.
\end{Theorem}
\begin{proof}
  The main idea of the proof is to construct a~linear Hamiltonian system in form \eqref{H} with the
  condition $\H(t)\geq0$ on $[a,b]$ for its coefficient matrix $\H(t)$ and such that $Y\nhspt Q$ is
  a~conjoined basis of this system for a~suitably chosen nonsingular matrix-valued function
  $Q:[a,b]\to\Rbb^{n\times n}$. Then by a~classical result for such systems, see e.g.
  \cite[Theorem~3]{wK2003}, \cite[Proof of Lemma~3.6(a)]{rF.rJ.cN2003}, or \cite[Theorem~2.4]{jvE.rSH2018},
  \cite[Theorems~1.79 and~1.81]{oD.jvE.rSH2019}, the condition $\H(t)\geq0$ will imply that the sets
  $\Ker X(t)\hspt Q(t)$ and $\Im X(t)$ are piecewise constant on $[a,b]$. Hence, the quantity
  $\rank X(t)$ is also piecewise constant on $[a,b]$, giving a~correct meaning to the right-hand sides
  of equations \eqref{Nge0} and \eqref{N*ge0}.

  A~construction of such Hamiltonian system is given in \cite{jvE.rSH2018}, see also
  \cite[Section~5.2.1]{oD.jvE.rSH2019}, where the authors considered symplectic spectral
  problems with self-adjoint boundary conditions depending on spectral parameter $\la\in\Rbb$. In
  the notation of this paper and according to the proof of \cite[Lemma~4.1]{jvE.rSH2018}, we consider the
  solution $Q(t)$ of the linear differential system
  \begin{equation} \label{E:Q.DE.def}
    Q'=-K_Y^2(t)\hspt Y^T\nhspt(t)\hspt Y'(t)\,Q, \quad t\in[a,b], \quad Q(a)=I.
  \end{equation}
  Since the coefficient matrix in \eqref{E:Q.DE.def} is piecewise continuous on $[a,b]$, it follows
  that the matrix $Q(t)$ is piecewise continuously differentiable and invertible on $[a,b]$. Consider
  the invertible matrix $M(t):=K_Y(t)\hspt Q^{T-1}(t)$ and modify the symplectic matrix $Z_Y(t)$ in
  \eqref{E:ZY.def} to become
  \begin{equation} \label{E:ZY.def.const.cor}
    \tilde Z_Y(t):=Z_Y(t)\diag\{M(t),M^{T-1}(t)\}
    =\mmatrix{\J\hspt Y(t)\hspt K_Y^2(t)\hspt Q^{T-1}(t) & Y(t)\hspt Q(t)}
  \end{equation}
  on $[a,b]$. Then the matrix $\tilde Z_Y(t)$ is piecewise continuously differentiable on $[a,b]$
  and symplectic, as a~product of two symplectic matrices. Therefore, the piecewise continuous matrix
  \begin{equation} \label{E:H.tZ.def}
    \H(t):=-\J\hspt[\tilde Z_Y(t)]\hspt'\hspt\tilde Z_Y^{-1}(t)
  \end{equation}
  is symmetric and satisfies with the above defined $\tilde Z_Y(t)$ the linear Hamiltonian system
  \begin{equation} \label{E:H.tZY}
    [\tilde Z_Y(t)]\hspt'=\J\hspt\H(t)\,\tilde Z_Y(t), \quad t\in[a,b].
  \end{equation}
  Note that $\H(t)=\Psi(\tilde Z_Y(t))$ with the notation from \cite[Propositions~2.2 and~2.3]{jvE.rSH2018}.
  Upon calculating the derivative of the matrix $\tilde Z_Y(t)$ in \eqref{E:ZY.def.const.cor}, we get
  (suppressing the argument $t$)
  \begin{equation} \label{E:tZY.diff}
    (\tilde Z_Y)'
    =\mmatrix{\J\hspt Y'K_Y^2\hspt Q^{T-1}\nhspt\nhspt+\nhspt\J\hspt Y(K_Y^2)'\hspt Q^{T-1}\nhspt\nhspt
      +\nhspt\J\hspt YK_Y^2\hspt (Q^{T-1})' & \quad Y'Q+YQ'}.
  \end{equation}
  Then by combining equations \eqref{E:Q.DE.def}--\eqref{E:tZY.diff} with \eqref{E:Y.Lagrangian.path.def}
  and assumption \eqref{derWr} we conclude that
  \begin{equation} \label{E:tZ.H.tZ}
    \tilde Z_Y^T(t)\hspt\H(t)\hspt\tilde Z_Y(t)
    =-V^T\nhspt (t)\,\diag\!\big\{Y^T\nhspt(t)\hspt\J\hspt Y'(t),\,Y^T\nhspt(t)\hspt\J\hspt Y'(t)\big\}\,V(t)
    \overset{\eqref{derWr}}{\geq}0
  \end{equation}
  on $[a,b]$, where the matrix $V(t):=\diag\{K_Y^2(t)\hspt Q^{T-1}(t),\,Q(t)\}$ is invertible. Note that
  we also used that the matrix $Q^{T-1}(t)$ solves on $[a,b]$ the adjoint system to \eqref{E:Q.DE.def},
  which is the linear differential system with the coefficient matrix equal to
  $[Y'(t)]^T\hspt Y(t)\hspt K_Y^2(t)$. Thus, according to \eqref{E:ZY.def.const.cor} and \eqref{E:tZ.H.tZ}
  the function $Y\nhspt Q$ is a~conjoined basis of the linear
  Hamiltonian system \eqref{E:H.tZY} with $\H(t)\geq0$ on $[a,b]$, so that the Legendre condition \eqref{E:LC}
  for this system holds. Applying \cite[Theorem~2.3]{jvE2019} we obtain that the oscillation number for
  the piecewise continuously differentiable Lagrangian path $Y\nhspt Q$ is equal to the total number of left
  proper focal points of $Y\nhspt Q$ in the interval $(a,b]$, i.e.,
  \begin{align*}
    \Nab{Y\nhspt Q} &= \sum_{t_0\in(a,b]} \!\!
      \big\{\!\rank\!\big(X(t_0^-)\hspt Q(t_0^-)\big)-\rank\!\big(X(t_0)\hspt Q_0(t_0)\big)\big\} \\
    &= \sum_{t_0\in(a,b]}\!\!\big(\nhspt\rank X(t_0^-)-\rank X(t_0)\big)\geq0.
  \end{align*}
  In a~similar way, by applying \cite[Proposition~3.3(v)]{jvE2020?c} we obtain that the dual oscillation
  number for $Y\nhspt Q$ is equal to the total number of right proper focal points of $Y\nhspt Q$ in the
  interval $[a,b)$, i.e.,
  \begin{align*}
    \Nsab{Y\nhspt Q} &= \sum_{t_0\in[a,b)} \!\!
      \big\{\!\rank\!\big(X(t_0^+)\hspt Q(t_0^+)\big)-\rank\!\big(X(t_0)\hspt Q_0(t_0)\big)\big\} \\
    &= \sum_{t_0\in[a,b)}\!\!\big(\nhspt\rank X(t_0^+)-\rank X(t_0)\big)\geq0.
  \end{align*}
  However, since the matrix $Q(t)$ is nonsingular on $[a,b]$, it follows from
  Remark~\ref{R:osc.number.multiple}(i) that $\Nab{Y\nhspt Q}=\Nab{Y}$ and $\Nsab{Y\nhspt Q}=\Nsab{Y}$, which
  completes the proof of this theorem.
\end{proof}

\begin{Remark} \label{R:pwc.kernel}
  The proof of Theorem~\ref{T:monoton.osc} shows that the conclusion of the piecewise constant image of
  $X(t)$ on $[a,b]$ in the statement of this theorem follows from the stronger property that the matrix
  $X(t)\hspt Q(t)$ has piecewise constant kernel on $[a,b]$, where $Q(t)$ is the solution of
  \eqref{E:Q.DE.def}.
\end{Remark}

By combining Theorems~\ref{T:osc.Maslov} and~\ref{T:monoton.osc} we derive the following monotonicity
property of the Maslov index of two continuously differentiable Lagrangian paths on $[a,b]$. Note that
here we do not make any strict monotonicity assumption.

\begin{Theorem} \label{T:monoton.osc.Maslov}
  Let $Y$ and\/ $\hY$ be piecewise continuously differentiable Lagrangian paths $[a,b]$. Assume that
  there exist a~piecewise continuously differentiable symplectic matrix $Z(t)$ and a~nonsingular
  piecewise continuously differentiable $n\times n$ matrix $P(t)$ such that
  \begin{equation} \label{derWrGen}
    [\hspt\bar Y'(t)]^T\!\J\hspt\bar Y(t)\geq0, \quad \bar Y(t):=Z^{-1}(t)\hspt\hY(t), \quad
    Z(t)\hspt E=Y(t)\hspt P(t), \quad t\in[a,b].
  \end{equation}
  Then the subspace $\Im W(Y(t),\hY(t))$ is piecewise constant on $[a,b]$ and
  \begin{align}
    \Masab{Y}{\hY} &=
      \sum_{t_0\in(a,b]}\!\big(\nhspt\rank W(Y(t_0^-),\hY(t_0^-))-\rank W(Y(t_0),\hY(t_0)\big)\ge 0,
      \label{NWge0} \\
    \Massab{Y}{\hY} &=
      \sum_{t_0\in[a,b)}\!\big(\nhspt\rank W(Y(t_0^+),\hY(t_0^+))-\rank W(Y(t_0),\hY(t_0)\big)\ge 0.
      \label{NW*ge0}
  \end{align}
  In particular, assumption \eqref{derWrGen} is satisfied under the conditions
  \begin{equation} \label{derWrGen.suff}
    [\hspt Y'(t)]^T\!\J\hspt Y(t)\leq 0, \quad [\hspt\hY'(t)]^T\!\J\hspt\hY(t)\geq0, \quad t\in[a,b].
  \end{equation}
\end{Theorem}
\begin{proof}
  The function $\bY$ defined in \eqref{derWrGen} is a~piecewise continuously differentiable Lagrangian
  path on $[a,b]$. Then by applying formula \eqref{E:osc.Maslov} in Theorem~\ref{T:osc.Maslov} and
  Remark~\ref{R:osc.Maslov1}(ii) we get
  \begin{align*}
    \Masab{Y}{\hY} &= \Masab{Y\nhspt P}{\hY}=\Nab{\bY}, \\
    \Massab{Y}{\hY} &= \Massab{Y\nhspt P}{\hY}=\Nsab{\bY}.
  \end{align*}
  For the calculation of these oscillation numbers we apply Theorem~\ref{T:monoton.osc} (with $Y:=\bY$).
  Note that the upper block $\bar X(t)$ of $\bar Y(t)$ has the form $\bar X(t)=-P^T(t)\,W(Y(t),\hY(t))$.
  Then, by the nonsingularity of the matrix $P(t)$, the results in \eqref{NWge0} and \eqref{NW*ge0}
  follow from equations \eqref{Nge0} and \eqref{N*ge0}.
  Finally, we assume that \eqref{derWrGen.suff} holds and consider the matrices $Q(t)$, $\tilde Z_Y(t)$,
  and $\H(t)$ given by \eqref{E:Q.DE.def}, \eqref{E:ZY.def.const.cor}, and \eqref{E:H.tZ.def}. Then we
  have (repeating the proof of Theorem~\ref{T:monoton.osc}) from the first condition in \eqref{derWrGen.suff}
  that $\H(t)\leq0$ (pay attention to the sign change) on $[a,b]$. Next we put $P(t):=Q(t)$ and
  $Z(t):=\tilde Z_Y(t)$ in \eqref{derWrGen}, so that
  \begin{align*}
    [\hspt\bar Y'(t)]^T\!\J\hspt\bar Y(t) &=
      \big(\!-\nhspt\tilde Z_Y^{-1}(t)\hspt[\tilde Z_Y(t)]\hspt'\hspt\tilde Z_Y^{-1}(t)\hspt\hY(t)
      +\tilde Z_Y^{-1}(t)\hspt\hY'(t)\big)^T\!\J\tilde Z_Y^{-1}(t)\hspt\hY(t) \\
    &= [\hspt\hY'(t)]^T\!\J\hspt\hY(t)-\hY^T(t)\hspt\H(t)\hspt\hY(t)\geq0
  \end{align*}
  on $[a,b]$, where we used that the matrix $\tilde Z_Y^{-1}(t)$ is symplectic and that the matrix
  $\H(t)$ is symmetric. The last inequality then follows from $\H(t)\leq0$ and the second condition in
  \eqref{derWrGen.suff}. The proof is complete.
\end{proof}

\begin{Remark} \label{R:strict.monotone}
  In \cite[Section~4]{pH.yL.aS2017} the authors consider monotonicity properties of the eigenvalues of the
  matrix $\GGamma(t)$ in \eqref{E:GGamma.def} in the following sense. As the parameter $t\in[a,b]$ varies
  in a~fixed direction, the eigenvalues of the matrix $\GGamma(t)$ move monotonically around unit circle
  $\Ubb$. These results demand strict monotonicity assumptions for the derivatives in \eqref{derWrGen} and
  \eqref{derWr}, see \cite[Lemma~4.2]{pH.yL.aS2017} in particular. In this case we would have the strict
  inequality $\H(t)>0$ on $[a,b]$ for the Hamiltonian of the differential system \eqref{E:H.tZY} associated
  with $Y(t)$ in the proof of Theorem~\ref{T:monoton.osc}. Therefore, in this case the system \eqref{E:H.tZY}
  would be completely controllable on $[a,b]$. The approach without a~strict monotonicity assumption and
  hence without the complete controllability condition, as  presented in \eqref{derWrGen.suff} above, is
  known in \cite{pS.rSH2020?e} in the analysis of proper focal points of conjoined bases of system \eqref{H}.
\end{Remark}

\section{Comparison theorems for Lagrangian paths} \label{S:Osc.compare}
In this section we derive Sturmian type comparison theorems for two continuous Lagrangian paths $Y$
and $\hY$ on $[a,b]$. Thus we extend the results in \cite[Theorem~4.4]{jvE2020a} and
\cite[Theorem~4.3]{jvE2020?c} to the context of arbitrary continuous Lagrangian paths on $[a,b]$. At the
same time we utilize the more general definition of the oscillation number and the dual oscillation
number as presented in Section~\ref{S:Osc.number}. Compared with the results in the latter two references,
which were proven by using the properties of the comparative index, we employ the results from
Section~\ref{S:Osc.number} based on the theory of Lidskii angles.

\begin{Theorem}[Comparison theorem for oscillation numbers] \label{T:compare.generalized}
  Let $Y$ and\/ $\hY$ be continuous Lagrangian paths on $[a,b]$. Then
  \begin{align}
    \Nab{Y}-\Nab{\hY} &= \mu(Y(b),\hY(b))-\mu(Y(a),\hY(a))+\Nab{Z_{\hY}^{-1}\hspt Y},
      \label{E:compare.osc.number} \\
    \Nsab{Y}-\Nsab{\hY} &=
      \mu^*(Y(a),\hY(a))-\mu^*(Y(b),\hY(b))+\Nsab{Z_{\hY}^{-1}\hspt Y},
      \label{E:compare.osc.number*}
  \end{align}
  where $Z_{\hY}(t)=Z_{\hY(t)}$ is the symplectic and orthogonal matrix defined
  in \eqref{E:ZY.def}, which is associated with the Lagrangian path $\hY(t)$.
  Moreover, the matrix $Z_{\hY}(t)$ in \eqref{E:compare.osc.number} and \eqref{E:compare.osc.number*}
  can be replaced by an arbitrary continuous symplectic matrix $\hZ(t)$ with $\hZ(t)\hspt E=\hY(t)$
  on $[a,b]$.
\end{Theorem}
\begin{proof}
  Applying Proposition~\ref{PJ:main.CI.Lidskii} (with the special case of $\tau_1:=a$ and
  $\tau_2:=b$) we derive by \eqref{EJ:mu.Lidskii} and \eqref{EJ:mu*.Lidskii} that
  \begin{align}
    \mu(Y(t),\hY(t))\big|_{a}^{b} &=
      \sum_{j=1}^n q_j(t)\big|_{a}^{b}
      -\sum_{j=1}^n{\hat q}_j(t)\big|_{a}^{b}
      -\sum_{j=1}^n{\tilde q}_j(t)\big|_{a}^{b}, \label{EJab:mu.Lidskii} \\
    \mu^*(Y(t),\hY(t))\big|_{a}^{b} &=
      -\sum_{j=1}^n q_j^*(t)\big|_{a}^{b}
      +\sum_{j=1}^n{\hat q}_j^*(t)\big|_{a}^{b}
      +\sum_{j=1}^n{\tilde q}_j^*(t)\big|_{a}^{b}, \label{EJab:mu*.Lidskii}
  \end{align}
  where according to Theorem~\ref{T:osc.number.q.q*} we have
  \begin{gather*}
    \Nab{Y}=\sum_{j=1}^n q_j(t)\big|_{a}^{b}, \quad
      \Nab{\hY}=\sum_{j=1}^n {\hat q}_j(t)\big|_{a}^{b}, \quad
      \Nab{Z_{\hY}^{-1}\hspt Y}=\sum_{j=1}^n{\tilde q}_j(t)\big|_{a}^{b}, \\
    \Nsab{Y}=\sum_{j=1}^n q_j^*(t)\big|_{a}^{b}, \quad
    \Nsab{\hY}=\sum_{j=1}^n {\hat q}_j^*(t)\big|_{a}^{b}, \quad
    \Nsab{Z_{\hY}^{-1}\hspt Y}=\sum_{j=1}^n{\tilde q}_j^*(t)\big|_{a}^{b}.
  \end{gather*}
  By substituting the last representations into formulas \eqref{EJab:mu.Lidskii} and
  \eqref{EJab:mu*.Lidskii} we complete the proofs of the results in \eqref{E:compare.osc.number} and
  \eqref{E:compare.osc.number*}. Moreover, by using \eqref{E:osc.invariance.transf} and
  \eqref{E:osc.invariance.transf*} in Theorem~\ref{T:osc.invariance.transf} (with $S(t):=\hZ(t)$
  and $S(t)\hspt E=\hY(t)$ on $[a,b]$) we see that the matrix $Z_{\hY}(t)$ in
  \eqref{E:compare.osc.number} and \eqref{E:compare.osc.number*} can be replaced by an~arbitrary
  continuous symplectic matrix $\hZ(t)$ with $\hZ(t)\hspt E=\hY(t)$ on $[a,b]$.
\end{proof}

If the Lagrangian paths $Y$ and $\hY$ have the same values at the endpoints of the interval
$[a,b]$, then we obtain from Theorem~\ref{T:compare.generalized} the following.

\begin{Corollary} \label{C:osc.number.compare}
  Let $Y$ and\/ $\hY$ be continuous Lagrangian paths on $[a,b]$ such that $Y(a)=\hY(a)$
  and $Y(b)=\hY(b)$. Then
  \begin{align}
    \Nab{Y}-\Nab{\hY} &= \Nab{Z_{\hY}^{-1}\hspt Y}, \label{E:compare.osc.number.same} \\
    \Nsab{Y}-\Nsab{\hY} &= \Nsab{Z_{\hY}^{-1}\hspt Y}, \label{E:compare.osc.number.same*}
  \end{align}
  where $Z_{\hY}(t)=Z_{\hY(t)}$ is the symplectic and orthogonal matrix defined
  in \eqref{E:ZY.def}, which is associated with the Lagrangian path $\hY(t)$.
  Moreover, the matrix $Z_{\hY}(t)$ in \eqref{E:compare.osc.number.same} and
  \eqref{E:compare.osc.number.same*} can be replaced by an arbitrary continuous symplectic
  matrix $\hZ(t)$ with $\hZ(t)\hspt E=\hY(t)$ on $[a,b]$.
\end{Corollary}
\begin{proof}
  Under the assumptions $Y(a)=\hY(a)$ and $Y(b)=\hY(b)$ the comparative indices and
  the dual comparative indices appearing in \eqref{E:compare.osc.number} and
  \eqref{E:compare.osc.number*} are zero. Hence, the results in \eqref{E:compare.osc.number.same}
  and \eqref{E:compare.osc.number.same*}, as well as the last statement of this corollary, follow
  from Theorem~\ref{T:compare.generalized}.
\end{proof}

Equations \eqref{E:compare.osc.number} and \eqref{E:compare.osc.number*} can be
simplified for special Lagrangian paths on $[a,b]$, for which the involved
comparative indices vanish. For this purpose we introduce the notation $\Ya$ and
$\Yb$ for continuous Lagrangian paths, which satisfy the initial conditions
\begin{equation} \label{E:Ya,Yb.ini}
  \Ya(a)=E=\Yb(b).
\end{equation}
When $\Ya$ and $\Yb$ correspond to conjoined bases of system \eqref{H}, then they are
uniquely determined by \eqref{E:Ya,Yb.ini} as solutions of \eqref{H}. In this case
$\Ya$ and $\Yb$ are called the {\em principal solutions\/} of system \eqref{H} at
the points $a$ and $b$. In the general situation of continuous Lagrangian paths on
$[a,b]$ we can still derive some important properties of the paths $\Ya$ and $\Yb$.
Then we have the following extension of \cite[Corollary~4.6]{jvE2020?c}.

\begin{Corollary} \label{C:osc.compare.Ya.Yb}
  Let $\Ya$ and\/ $\Yb$ be continuous Lagrangian paths on $[a,b]$ satisfying
  condition \eqref{E:Ya,Yb.ini}. Then their oscillation numbers and dual
  oscillation numbers satisfy the relations
  \begin{align}
    \Nab{\Yb}-\Nsab{\Ya} &= \Nab{Z_{\Ya}^{-1}\Yb},
      \label{E:osc.compare.Ya.Yb} \\
    \Nsab{\Yb}-\Nab{\Ya} &= \Nsab{Z_{\Ya}^{-1}\Yb}.
      \label{E:osc.compare.Ya.Yb*}
  \end{align}
\end{Corollary}
\begin{proof}
  We partition the Lagrangian paths $\Ya$ and $\Yb$ on $[a,b]$ according to \eqref{E:JYH.partition}.
  We apply equality \eqref{E:compare.osc.number} with $Y:=\Yb$ and $\hY:=\Ya$. Then
  $\mu(\Yb(b),\Ya(b))=\rank\Xa(b)$ and $\mu(\Yb(a),\Ya(a))=0$, so that by
  Theorem~\ref{T:osc.dual.osc.number} (with $Y:=\Ya$) we obtain
  \begin{align*}
    \Nab{\Yb} &\overset{\eqref{E:compare.osc.number}}{=}
      \Nab{\Ya}+\rank\Xa(b)+\Nab{Z_{\Ya}^{-1}\Yb} \\
    &\hspace*{-0.6mm}\overset{\eqref{E:osc.dual.osc.number}}{=}
      \Nsab{\Ya}+\Nab{Z_{\Ya}^{-1}\Yb}.
  \end{align*}
  This shows \eqref{E:osc.compare.Ya.Yb}. Next we apply \eqref{E:compare.osc.number*}
  with $Y:=\Yb$ and $\hY:=\Ya$. Then $\mu^*(\Yb(a),\Ya(a))=0$ and
  $\mu^*(\Yb(b),\Ya(b))=\rank\Xa(b)$, so that by Theorem~\ref{T:osc.dual.osc.number}
  (with $Y:=\Ya$) we get
  \begin{align*}
    \Nsab{\Yb} &\overset{\eqref{E:compare.osc.number*}}{=}
      \Nsab{\Ya}-\rank\Xa(b)+\Nsab{Z_{\Ya}^{-1}\Yb} \\
    &\hspace*{-0.7mm} \overset{\eqref{E:osc.dual.osc.number}}{=}
      \Nab{\Ya}+\Nsab{Z_{\Ya}^{-1}\Yb}.
  \end{align*}
  This shows \eqref{E:osc.compare.Ya.Yb*} and the proof is complete.
\end{proof}

Based on the connections of the oscillation number and the Maslov index derived in
Section~\ref{S:Maslov} we can now reformulate the comparison theorem for the oscillation numbers
(Theorem~\ref{T:compare.generalized}) in terms of the Maslov index. More precisely, we obtain
a~formula calculating the Maslov index $\Masab{Y}{\hY}$ in terms of the two reference
Maslov indices $\Masab{E}{\hY}$ and $\Masab{E}{Y}$ and in terms of the comparative
index of $\hY$ and $Y$ evaluated at the endpoints of $[a,b]$.

\begin{Corollary} \label{C:compare.Maslov}
  Let $Y$ and\/ $\hY$ be continuous Lagrangian paths on $[a,b]$. Then we have
  \begin{equation} \label{E:compare.Maslov}
    \left. \begin{array}{rl}
      \Masab{Y}{\hY} &\!\!\!= \Masab{E}{\hY}-\Masab{E}{Y} \\[1mm]
      &\hspace*{15mm} +\,\mu(\hY(a),Y(a))-\mu(\hY(b),Y(b)).
    \end{array} \!\right\}
  \end{equation}
  If in addition $Y(a)=\hY(a)$ and $Y(b)=\hY(b)$ hold, then
  \begin{equation} \label{E:compare.Maslov.same}
    \Masab{Y}{\hY}=\Masab{E}{\hY}-\Masab{E}{Y}.
  \end{equation}
\end{Corollary}
\begin{proof}
  Let $Y$ and $\hY$ be continuous Lagrangian paths on $[a,b]$. We apply formula
  \eqref{E:compare.osc.number} in Theorem~\ref{T:compare.generalized}, in which
  we interchange the roles of $Y$ and $\hY$. Then we get
  \begin{equation} \label{E:compare.Maslov.hlp1}
    \Nab{\hY}-\Nab{Y}=\mu(\hY(b),Y(b))-\mu(\hY(a),Y(a))+\Nab{Z_Y^{-1}\hspt\hY}.
  \end{equation}
  If we now replace the oscillation numbers appearing in \eqref{E:compare.Maslov.hlp1}
  by the corresponding Maslov indices from \eqref{E:osc.number.Maslov.C} and
  \eqref{E:osc.Maslov}, then we obtain the result in \eqref{E:compare.Maslov}.
  Finally, equation \eqref{E:compare.Maslov.same} follows from identity \eqref{E:compare.Maslov},
  in which the comparative indices vanish under the assumptions that $Y(a)=\hY(a)$ and
  $Y(b)=\hY(b)$.
\end{proof}

\begin{Remark}\label{R:compar:maslov.dual}
  Combining equations \eqref{E:osc.Maslov*} and \eqref{E:osc.number.Maslov.C*} with
  the comparison theorem for the dual oscillation numbers in \eqref{E:compare.osc.number*} we
  derive a~dual version of Corollary~\ref{C:compare.Maslov} in the form
  \begin{equation} \label{E:compare.Maslov*}
    \left. \begin{array}{rl}
      \Massab{Y}{\hY} &\!\!\!= \Massab{E}{\hY}-\Massab{E}{Y} \\[1mm]
      &\hspace*{15mm} +\,\mu^*(\hY(b),Y(b))-\mu^*(\hY(a),Y(a)).
    \end{array} \!\right\}
  \end{equation}
  In addition, if $Y(a)=\hY(a)$ and $Y(b)=\hY(b)$ hold, then \eqref{E:compare.Maslov*} reduces to
  \begin{equation} \label{E:compare.Maslov.same*}
    \Massab{Y}{\hY}=\Massab{E}{\hY}-\Massab{E}{Y}.
  \end{equation}
  Comparison formulas \eqref{E:compare.Maslov} and \eqref{E:compare.Maslov*} appear to be new in
  the context of the Maslov index.
\end{Remark}

Next we consider the corresponding Sturmian type separation theorems for the
oscillation numbers and the dual oscillation numbers. These results extend the
special situation, when both Lagrangian paths $Y$ and $\hY$ are conjoined bases of one
linear Hamiltonian system \eqref{H}. We refer to \cite[Theorems~2.2 and~2.3]{jvE2016}
and \cite[Theorem~4.1]{pS.rSH2017} for the case when the Legendre condition \eqref{E:LC}
holds, and to \cite[Theorem~4.1]{jvE2020} and \cite[Theorem~4.4]{jvE2020?c} for the
case without assumption \eqref{E:LC}.

Let us fix a~continuous symplectic matrix $\Phi(t)$ on $[a,b]$. Consider the set
\begin{equation} \label{E:set.F.def}
  \F(\Phi):=\big\{ \Phi(\cdot)\hspt C, \text{ where $C$ is a~Lagrangian plane}
    \big\}.
\end{equation}
Then the elements $Y\in\F(\Phi)$ are continuous Lagrangian paths on $[a,b]$, which are
constant multiples of the given symplectic matrix $\Phi(t)$ on $[a,b]$, that is,
$Y(t)=\Phi(t)\,C$ on $[a,b]$ for some matrix $C\in\Rbb^{2n\times n}$ with $C^T\!\J C=0$
and $\rank C=n$. In the context of system \eqref{H} the set $\F(\Phi)$ corresponds to the
set of all conjoined bases of \eqref{H}. With the notation in \eqref{E:set.F.def} we can
formulate the following.

\begin{Theorem}[Separation theorem] \label{T:osc.number.separ}
  Let $\Phi(t)$ be a~continuous symplectic matrix on $[a,b]$. For any continuous
  Lagrangian paths $Y$ and $\hY$ belonging to the set $\F(\Phi)$ defined in
  \eqref{E:set.F.def} we have
  \begin{align}
    \Nab{Y}-\Nab{\hY} &= \mu(Y(b),\hY(b))-\mu(Y(a),\hY(a)),
      \label{E:separ.osc.number} \\
    \Nsab{Y}-\Nsab{\hY} &= \mu^*(Y(a),\hY(a))-\mu^*(Y(b),\hY(b)).
      \label{E:separ.osc.number*}
  \end{align}
\end{Theorem}
\begin{proof}
  Formulas \eqref{E:separ.osc.number} and \eqref{E:separ.osc.number*} follow from the comparison theorem
  (Theorem~\ref{T:compare.generalized}) and from Corollary~\ref{C:osc.number.const.rank}. Indeed,
  considering the continuous Lagrangian path $\tY:=\hZ^{-1}Y$ on $[a,b]$, then its upper block
  $\tX$ has the form $\tX(t)=-W(\hY(t),Y(t))$, which is a~constant matrix on $[a,b]$ in the setting
  of this theorem. Hence, we have $\Nab{\tY}=0$ and $\Nsab{\tY}=0$ by
  Corollary~\ref{C:osc.number.const.rank} and then equations \eqref{E:compare.osc.number} and
  \eqref{E:compare.osc.number*} yield the results in \eqref{E:separ.osc.number} and
  \eqref{E:separ.osc.number*}.
\end{proof}

The results in Theorem~\ref{T:osc.number.separ} show that for a~given continuous
Lagrangian path $Y\in\F(\Phi)$ it is possible to calculate the value $\Nab{Y}$
from the oscillation number of a~suitable reference Lagrangian path from the set
$\F(\Phi)$ in \eqref{E:set.F.def}.

\begin{Remark} \label{R:osc.number.Ya.Yb}
  Consider the continuous Lagrangian paths $\Ya,\Yb\in\F(\Phi)$, which are associated in
  \eqref{E:set.F.def} with the matrices $C_a:=\Phi^{-1}(a)\hspt E$ and
  $C_b:=\Phi^{-1}(b)\hspt E$, i.e.,
  \begin{equation} \label{E:Ya.Yb.Phi.ini}
    \Ya(t)=\Phi(t)\,C_a, \quad \Yb(t)=\Phi(t)\,C_b, \quad t\in[a,b],
    \quad \Ya(a)=E=\Yb(b).
  \end{equation}
  Since for any Lagrangian path $Y\in\F(\Phi)$ we have $\mu(Y(a),E)=0$ and
  $\mu^*(Y(b),E)=0$ by \eqref{E:comparative.index.def}, it follows from
  \eqref{E:separ.osc.number} with $\hY:=\Ya$ and from \eqref{E:separ.osc.number*}
  with $\hY:=\Yb$ that
  \begin{align}
    \Nab{Y} &= \Nab{\Ya}+\mu(Y(b),\Ya(b)), \label{E:separ.osc.number.Ya} \\
    \Nsab{Y} &= \Nsab{\Yb}+\mu^*(Y(a),\Yb(a)). \label{E:separ.osc.number*.Yb}
  \end{align}
  In addition, in the same spirit as in \cite[Theorem~4.5]{jvE2020?c} we obtain from
  \eqref{E:separ.osc.number.Ya} and \eqref{E:separ.osc.number*.Yb} for any Lagrangian
  path $Y\in\F(\Phi)$ the estimates
  \begin{align}
    \Nab{\Ya} &\leq \Nab{Y}\leq\Nab{\Yb}, \label{E:separ.osc.estimate} \\
    \Nsab{\Yb} &\leq \Nsab{Y}\leq\Nsab{\Ya}. \label{E:separ.osc.estimate*}
  \end{align}
  Moreover, with the special choices of $Y:=\Yb$ in \eqref{E:separ.osc.number.Ya} and
  $Y:=\Ya$ in \eqref{E:separ.osc.number*.Yb} we deduce that
  \begin{gather}
    \Nab{\Ya}=\Nsab{\Yb}, \quad \Nab{\Yb}=\Nsab{\Ya}, \label{E:separ.osc.equal.Ya.Yb} \\
    \Nab{\Yb}-\Nab{\Ya}=\rank W(\Ya,Y_b)=\Nsab{\Ya}-\Nsab{\Yb},
      \label{E:separ.osc.equal.Ya.Yb.rank}
  \end{gather}
  compare with \cite[Corollary~5.4 and Theorem~5.6]{pS.rSH2017}. Note that in order
  to derive \eqref{E:separ.osc.equal.Ya.Yb} and \eqref{E:separ.osc.equal.Ya.Yb.rank}
  from \eqref{E:separ.osc.number.Ya} and \eqref{E:separ.osc.number*.Yb} we used that
  $\mu(E,\Ya(b))$ and $\mu^*(E,\Yb(a))$ have the same value, since the Wronskian
  $W(\Ya(t),Y_b(t))\equiv W(\Ya,Y_b)$ is in this case constant on $[a,b]$.
\end{Remark}

Based on the above separation theorem we are able to answer the question about the
existence of a~continuous Lagrangian path $Y$ in the given set $\F(\Phi)$, whose
oscillation number and dual oscillation number attain prescribed values satisfying
estimates \eqref{E:separ.osc.estimate} and \eqref{E:separ.osc.estimate*}. This
result generalizes \cite[Theorem~1.1]{pS.rSH2021} to arbitrary continuous Lagrangian
paths or even to conjoined bases of system \eqref{H} without assuming the Legendre
condition \eqref{E:LC}.

\begin{Theorem} \label{T:distrib.osc.number}
  Let $\Phi(t)$ be a~continuous symplectic matrix on $[a,b]$ and let $\Ya,\Yb\in\F(\Phi)$
  be the unique continuous Lagrangian paths determined by formula \eqref{E:Ya.Yb.Phi.ini}
  in Remark~\ref{R:osc.number.Ya.Yb}. Then for any integers $\ell$ and $r$ satisfying
  \begin{equation} \label{E:main.number.lr.assume}
    \Nab{\Ya}\leq\ell\leq\Nab{\Yb} \qtextq{and} \Nsab{\Yb}\leq r \leq\Nsab{\Ya}
  \end{equation}
  there exists a~continuous Lagrangian path $Y\in\F(\Phi)$ such that
  \begin{equation} \label{E:mL.mR.lr}
    \Nab{Y}=\ell \qtextq{and} \Nsab{Y}=r.
  \end{equation}
  Moreover, if\/ $\ell\geq r$, then the Lagrangian path $Y$ can be chosen with $X(a)=I$,
  and if\/ $\ell\leq r$, then the Lagrangian path $Y$ can be chosen with $X(b)=I$. In
  particular, when $\ell=r$ the Lagrangian path $Y$ may be chosen with both $X(a)$
  and $X(b)$ invertible.
\end{Theorem}
\begin{proof}
  Let $\ell$ and $r$ be given integers satisfying \eqref{E:main.number.lr.assume}.
  Define the integers
  \begin{equation} \label{E:main.number.lr.hlp1}
    p:=\ell-\Nab{\Ya}, \quad q:=r-\Nsab{\Yb}.
  \end{equation}
  Then from \eqref{E:main.number.lr.assume} and \eqref{E:separ.osc.equal.Ya.Yb.rank}
  it follows that $\max\{p,q\}\leq w:=\rank W(\Ya,\Yb)$ holds, where the Wronskian is
  constant on $[a,b]$ by the last part of Remark~\ref{R:osc.number.Ya.Yb}. We will
  apply the ideas of \cite[Theorem~2.1]{pS.rSH2021}, which we adopt to the setting
  of continuous Lagrangian paths in the set $\F(\Phi)$. We partition the
  Lagrangian paths $\Ya=(\Xa^T,\Ua^T)^T$ and $\Yb=(\Xb^T,\Ub^T)^T$ according to
  notation \eqref{E:JYH.partition}. First we assume that $\ell\geq r$ holds. In view of
  \eqref{E:separ.osc.equal.Ya.Yb} we then obtain that $p\geq q$. In the spirit of the
  proof of \cite[Theorem~2.1]{pS.rSH2021} we construct the Lagrangian path
  $Y\in\F(\Phi)$ by
  \begin{equation} \label{E:distrib.osc.number.hlp2}
    Y(t)=\mmatrix{X(t) \\ U(t)}:=\Phi(t)\,C, \quad t\in[a,b], \quad
    C:=\Phi^{-1}(a)
     \mmatrix{I \\ D+R_b(a)\hspt\Ub(a)\hspt\Xb^\dagger(a)},
  \end{equation}
  where $R_b(a):=\Xb(a)\hspt\Xb^\dagger(a)$ is the orthogonal projector onto
  $\Im\Xb(a)$ and where the symmetric $n\times n$ matrix $D$ has $q$ negative
  eigenvalues $\la_j=-1$ and $w-p$ positive eigenvalues $\la_j=1$. More precisely,
  see \cite[Eqs.~(2.29)--(2.31)]{pS.rSH2021}, we take
  \begin{equation} \label{E:main.number.pq.hlp4}
    D:= L\diag\{-I_q,I_{w-p},0_{n-w+p-q}\}\hspt L^T,
  \end{equation}
  where $L$ is an~orthogonal matrix satisfying
  \begin{equation} \label{E:main.number.pq.hlp2}
    R_b(a)=L\diag\{I_w,0_{n-w}\}\hspt L^T, \quad
    R_b(a)\hspt DR_b(a)=D.
  \end{equation}
  Equations \eqref{E:main.number.pq.hlp4} and \eqref{E:main.number.pq.hlp2} imply,
  compare with \cite[Lemma~2.2]{pS.rSH2021}, that
  \begin{align}
    \mu(Y(b),\Ya(b)) &=
      w-\ind\hspt[\hspt-R_b(a)\hspt DR_b(a)]
      \overset{\eqref{E:main.number.pq.hlp2}}{=} w-\ind\hspt(-D)
      \overset{\eqref{E:main.number.pq.hlp4}}{=} w-(w-p)=p,
      \label{E:distrib.osc.number.hlp3} \\
    \mu^*(Y(a),\Yb(a)) &= \ind\hspt[\hspt R_b(a)\hspt DR_b(a)]
      \overset{\eqref{E:main.number.pq.hlp2}}{=} \ind D
      \overset{\eqref{E:main.number.pq.hlp4}}{=} q.
      \label{E:distrib.osc.number.hlp4}
  \end{align}
  Then by \eqref{E:separ.osc.number.Ya} and \eqref{E:separ.osc.number*.Yb} we obtain that
  \begin{align*}
    \Nab{Y} &\overset{\eqref{E:separ.osc.number.Ya}}{=} \Nab{\Ya}+\mu(Y(b),\Ya(b))
      \overset{\eqref{E:distrib.osc.number.hlp3}}{=}
      \Nab{\Ya}+p \overset{\eqref{E:main.number.lr.hlp1}}{=} \ell, \\
    \Nsab{Y} &\overset{\eqref{E:separ.osc.number*.Yb}}{=} \Nsab{\Yb}+\mu^*(Y(a),\Yb(a))
      \overset{\eqref{E:distrib.osc.number.hlp4}}{=}
      \Nsab{\Yb}+q \overset{\eqref{E:main.number.lr.hlp1}}{=} r.
  \end{align*}
  This completes the proof of \eqref{E:mL.mR.lr} for the case of $\ell\geq r$. Moreover,
  from \eqref{E:distrib.osc.number.hlp2} with $t=a$ we can see that $X(a)=I$ holds.
  Next we suppose that $\ell\leq r$, so that $p\leq q$ in view of
  \eqref{E:separ.osc.equal.Ya.Yb}. We follow the second part of the proof of
  \cite[Theorem~2.1]{pS.rSH2021}. Hence, we construct the Lagrangian path
  $Y\in\F(\Phi)$ by the formula
  \begin{equation} \label{E:distrib.osc.number.hlp5}
    Y(t)=\mmatrix{X(t) \\ U(t)}:=\Phi(t)\,C, \quad t\in[a,b], \quad
    C:=\Phi^{-1}(b)
     \mmatrix{I \\ D+R_a(b)\hspt\Ua(b)\hspt\Xa^\dagger(b)},
  \end{equation}
  where $R_a(b):=\Xa(b)\hspt\Xa^\dagger(b)$ is the orthogonal projector onto
  $\Im\Xa(b)$ and where as in \cite[Eq.~(2.33)]{pS.rSH2021} the symmetric $n\times n$
  matrix $D$ has the form
  \begin{equation} \label{E:main.number.pq.hlp6}
    D:= L\diag\{-I_{w-q},I_p,0_{n-w+q-p}\}\hspt L^T
  \end{equation}
  with an~orthogonal matrix $L$ satisfying
  \begin{equation} \label{E:main.number.pq.hlp7}
    R_a(b)=L\diag\{I_w,0_{n-w}\}\hspt L^T, \quad
    R_a(b)\hspt DR_a(b)=D.
  \end{equation}
  Equations \eqref{E:main.number.pq.hlp6} and \eqref{E:main.number.pq.hlp7} imply,
  compare with \cite[Lemma~2.3]{pS.rSH2021}, that
  \begin{align}
    \mu(Y(b),\Ya(b)) &= \ind\hspt[\hspt -R_a(b)\hspt DR_a(b)]
      \overset{\eqref{E:main.number.pq.hlp7}}{=} \ind\hspt(-D)
      \overset{\eqref{E:main.number.pq.hlp6}}{=} p,
      \label{E:distrib.osc.number.hlp8} \\
    \mu^*(Y(a),\Yb(a)) &= w-\ind\hspt[\hspt R_a(b)\hspt DR_a(b)]
      \overset{\eqref{E:main.number.pq.hlp7}}{=} w-\ind D
      \overset{\eqref{E:main.number.pq.hlp6}}{=} w-(w-q)=q.
      \label{E:distrib.osc.number.hlp9}
  \end{align}
  Then by \eqref{E:separ.osc.number.Ya} and \eqref{E:separ.osc.number*.Yb} we obtain that
  \begin{align*}
    \Nab{Y} &\overset{\eqref{E:separ.osc.number.Ya}}{=} \Nab{\Ya}+\mu(Y(b),\Ya(b))
      \overset{\eqref{E:distrib.osc.number.hlp8}}{=}
      \Nab{\Ya}+p \overset{\eqref{E:main.number.lr.hlp1}}{=} \ell, \\
    \Nsab{Y} &\overset{\eqref{E:separ.osc.number*.Yb}}{=} \Nsab{\Yb}+\mu^*(Y(a),\Yb(a))
      \overset{\eqref{E:distrib.osc.number.hlp9}}{=}
      \Nsab{\Yb}+q \overset{\eqref{E:main.number.lr.hlp1}}{=} r.
  \end{align*}
  This completes the proof of \eqref{E:mL.mR.lr} for the case of $\ell\leq r$. Moreover,
  from \eqref{E:distrib.osc.number.hlp5} with $t=b$ we can see that $X(b)=I$ holds.
  Finally, if $\ell=r$, then from \eqref{E:mL.mR.lr} we have $\Nab{Y}=\Nsab{Y}$.
  By formula \eqref{E:osc.dual.osc.number} in Theorem~\ref{T:osc.dual.osc.number} it
  follows that $\rank X(a)=\rank X(b)$. This means that if $X(a)=I$, then $X(b)$ is
  nonsingular, while if $X(b)=I$, then $X(a)$ is nonsingular. In conclusion, if $\ell=r$,
  then the Lagrangian path $Y$ can be chosen with both $X(a)$ and $X(b)$ invertible. The
  proof is complete.
\end{proof}

\begin{Remark} \label{R:distrib.ini}
  The proof of Theorem~\ref{T:distrib.osc.number} shows that the Lagrangian path
  $Y\in\F(\Phi)$ satisfying condition \eqref{E:mL.mR.lr} is constructed on $[a,b]$ as
  a~constant multiple of the matrix $\Phi(t)$ by prescribing its initial condition
  at $a$ in \eqref{E:distrib.osc.number.hlp2} if $\ell\geq r$, or at $b$ in
  \eqref{E:distrib.osc.number.hlp5} if $\ell\leq r$.
\end{Remark}

\section{Conclusions} \label{S:Conclude}
In this section we make comments about the main results of this paper and their
relationship with some related mathematical problems. The main purpose of this paper
was to study the oscillation number $\Nab{Y}$ and the dual oscillation number
$\Nsab{Y}$ for a~continuous Lagrangian path $Y$ on $[a,b]$. These are integer
quantities defined in an~algebraic way through the comparative index and the dual
comparative index by using a~certain partition of the interval $[a,b]$. Here we use
a~more general definition than in \cite{jvE2020,jvE2020?c} in the sense that we use
nonconstant symplectic matrices in the partition.
As the main results (Theorem~\ref{T:osc.number.q.q*} and the subsequent results in
Section~\ref{S:Osc.compare}) we express the quantities $\Nab{Y}$ and $\Nsab{Y}$ in
terms of the total changes in the interval $[a,b]$ of the integers $q_j(t)$ and
$q_j^*(t)$, which are associated through \eqref{E:qjt.def} and \eqref{E:qjt*.def}
with the continuous Lidskii angles $\varphi_j(t)$ of the symplectic and orthogonal
matrix $Z_Y(t)$ on $[a,b]$ defined in \eqref{E:ZY.def}. The methods, which were used
for the above analysis, are based on the Lidskii angles of symplectic matrices and
their relationship with the comparative index obtained recently in \cite{pS.rSH2020?g}.
This approach allowed us to connect the oscillation numbers with the Maslov index
(Theorems~\ref{T:osc.Maslov} and~\ref{T:monoton.osc}, Corollary~\ref{C:osc.number.Maslov}, and
Remarks~\ref{R:dual.Maslov.index} and~\ref{R:compar:maslov.dual}). In addition, we
derive a~general comparison theorem (Theorem~\ref{T:compare.generalized}) for the
oscillation numbers and the dual oscillation
numbers of two arbitrary continuous Lagrangian paths $Y$ and $\hY$ on $[a,b]$, as well
as general separation theorems for the case when $Y$ and $\hY$ are constant multiples
of a~given continuous symplectic matrix (Theorems~\ref{T:osc.number.separ}
and~\ref{T:distrib.osc.number}).

The theory presented in this paper is based on detailed matrix analysis with motivations
coming from the theory of differential equations, resp. from the oscillation theory of
linear Hamiltonian systems \eqref{H}. It can be further developed in the direction of
the oscillation theory on discrete time domains \cite{oD.jvE.rSH2019,pS.rSH2018a} or
in the direction of singular comparison theorems for the oscillation numbers, as we
recently presented in \cite{pS.rSH2019,pS.rSH2020} for conjoined bases of nonoscillatory
linear Hamiltonian systems. Such results may have fundamental applications in the theory
of Maslov index on unbounded intervals, or in the study of the rotation number of
a~family of linear Hamiltonian systems, such as in \cite{rJ.sN.cN.rO2017,rJ.rO.sN.cN.rF2016}.
We also expect related research activity in the theory of matrices in general, for example
in the limit theorems for symmetric matrix valued functions generalizing the results in
\cite{wK1993,wK.rSH2013} or \cite[Theorem~3.3.7]{wK1995}.





\begin{thebibliography}{99}
\newcommand{\mciteb}[4]{{\rm #1, }{\it #2,\ }{\rm #3.}{\ #4}}
\newcommand{\mcitej}[4]{{\rm #1, }{\rm #2,\ }{\rm #3.}{\ #4}}
\newcommand{\jour}[1]{{\it #1}}

\bibitem{aaA2001}
\mcitej{A.~A.~Abramov}
       {On the computation of the eigenvalues of a nonlinear spectral problem for
        Hamiltonian systems of ordinary differential equations}
       {\jour{Zh. Vychisl. Mat. Mat. Fiz.} {\bf 41} (2001), no.~1, 29--38; translation in
        \jour{Comput. Math. Math. Phys.} {\bf 41} (2001), no.~1, 27--36}
       {}

\bibitem{aaA2011}
\mcitej{A.~A.~Abramov}
       {A modification of one method for solving nonlinear self-adjoint eigenvalue
        problem for Hamiltonian systems of ordinary differential equations}
       {\jour{Zh. Vychisl. Mat. Mat. Fiz.} {\bf 51} (2011), no.~1, 39--43; translation in
        \jour{Comput. Math. Math. Phys.} {\bf 51} (2011), no.~1, 35--39}
       {}

\bibitem{fvA1964}
\mciteb{F.~V.~Atkinson}
       {Discrete and Continuous Boundary Problems}
       {Academic Press, New York - London, 1964}
       {}

\bibitem{aB.tneG2003}
\mciteb{A.~Ben-Israel, T.~N.~E.~Greville}
       {Generalized Inverses: Theory and Applications}
       {Second Edition, Springer-Verlag, New York, NY, 2003}
       {}

\bibitem{bBB.kF1998}
\mcitej{B.~Booss-Bavnbek, K.~Furutani}
       {The Maslov index: a~functional analytical definition and the spectral flow
        formula}
       {\jour{Tokyo J. Math.} {\bf 21} (1998), no.~1, 1--34}
       {}

\bibitem{bBB.cZ2018}
\mcitej{B.~Booss-Bavnbek, C.~Zhu}
       {The Maslov index in symplectic Banach spaces}
       {\jour{Mem. Amer. Math. Soc.} {\bf 252} (2018), no.~1201, x+118 pp}
       {}

\bibitem{slC.cdM2009}
\mciteb{S.~L.~Campbell, C.~D.~Meyer}
       {Generalized Inverses of Linear Transformations}
       {Reprint of the 1991 corrected reprint of the 1979 original, Classics in Applied
        Mathematics, Vol.~56, Society for Industrial and Applied Mathematics (SIAM),
        Philadelphia, PA, 2009}
       {}

\bibitem{seC.rL.eyM1994}
\mcitej{S.~E.~Cappell, R.~Lee, E.~Y.~Miller}
       {On the Maslov index}
       {\jour{Comm. Pure Appl. Math.} {\bf 47} (1994), no.~2, 121--186}
       {}

\bibitem{oD2017a}
\mcitej{O.~Do\v{s}l\'y}
       {Relative oscillation of linear Hamiltonian differential systems}
       {\jour{Math. Nachr.} {\bf 290} (2017), no.~14--15, 2234--2246}
       {}

\bibitem{oD.jvE.rSH2019}
\mciteb{O.~Do\v{s}l\'y, J.~V.~Elyseeva, R.~\v{S}imon~Hilscher}
       {Symplectic Difference Systems: Oscillation and Spectral Theory}
       {Pathways in Mathematics, Birkh\"auser/Springer, Cham, 2019}
       {}

\bibitem{jvE2007}
\mcitej{J.~V.~Elyseeva}
       {The comparative index for conjoined bases of symplectic difference systems}
       {in: ``Difference Equations, Special Functions, and Orthogonal
        Polynomials'', Proceedings of the International Conference (Munich, 2005),
        S.~Elaydi, J.~Cushing, R.~Lasser, A.~Ruffing, V.~Papageorgiou, and W.~Van Assche,
        editors, pp.~168--177, World Scientific, London, 2007}
       {}

\bibitem{jvE2009a}
\mcitej{J.~V.~Elyseeva}
       {Comparative index for solutions of symplectic difference systems}
       {\jour{Differential Equations} {\bf 45} (2009), no.~3, 445--459; translated from
        \jour{Differencial'nyje Uravnenija} {\bf 45} (2009), no.~3, 431--444}
       {}

\bibitem{jvE2016}
\mcitej{J.~V.~Elyseeva}
       {Comparison theorems for conjoined bases of linear Hamiltonian differential systems
        and the comparative index}
       {\jour{J. Math. Anal. Appl.} {\bf 444} (2016), no.~2, 1260--1273}
       {}

\bibitem{jvE2017}
\mcitej{J.~V.~Elyseeva}
       {On symplectic transformations of linear Hamiltonian differential systems without
        normality}
       {\jour{Appl. Math. Lett.} {\bf 68} (2017), 33--39}
       {}

\bibitem{jvE2018}
\mcitej{J.~V.~Elyseeva}
       {The comparative index and transformations of linear Hamiltonian differential
        systems}
       {\jour{Appl. Math. Comput.} {\bf 330} (2018), 185--200}
       {}

\bibitem{jvE.rSH2018}
\mcitej{J.~V.~Elyseeva, R.~\v{S}imon~Hilscher}
       {Discrete oscillation theorems for symplectic eigenvalue problems with general boundary conditions
        depending nonlinearly on spectral parameter}
       {\jour{Linear Algebra Appl.} {\bf 558} (2018), 108--145}
       {}

\bibitem{jvE2019}
\mcitej{J.~V.~Elyseeva}
       {Oscillation theorems for linear Hamiltonian systems with nonlinear dependence on
        the spectral parameter and the comparative index}
       {\jour{Appl. Math. Lett.} {\bf 90} (2019), 15--22}
       {}

\bibitem{jvE2020}
\mcitej{J.~V.~Elyseeva}
       {Relative oscillation of linear Hamiltonian differential systems without
        monotonicity}
       {\jour{Appl. Math. Lett.} {\bf 103} (2020), Article 106173, 8 pp}
       {}

\bibitem{jvE2020a}
\mcitej{J.~V.~Elyseeva}
       {Comparison theorems for conjoined bases of linear Hamiltonian systems without
        monotonicity}
       {\jour{Monatsh. Math.} {\bf 193} (2020), no.~2, 305--328}
       {}

\bibitem{jvE2020b}
\mcitej{J.~V.~Elyseeva}
       {Renormalized oscillation theory for symplectic eigenvalue problems with nonlinear
        dependence on the spectral parameter}
       {\jour{J. Difference Equ. Appl.} {\bf 26} (2020), no.~4, 458--487}
       {}

\bibitem{jvE2020?c}
\mcitej{J.~V.~Elyseeva}
       {Relative oscillation theory for linear Hamiltonian systems with nonlinear
        dependence on the spectral parameter}
       {submitted (2020)}
       {}

\bibitem{rF.rJ.sN.cN2011}
\mcitej{R.~Fabbri, R.~Johnson, S.~Novo, C.~N\'u\~nez}
       {Some remarks concerning weakly disconjugate linear Hamiltonian systems}
       {\jour{J. Math. Anal. Appl.} {\bf 380} (2011), no.~2, 853--864}
       {}

\bibitem{rF.rJ.cN2003}
\mcitej{R.~Fabbri, R.~Johnson, C.~N\'u\~nez}
       {On the Yakubovich frequency theorem for linear non-autonomous control processes}
       {\jour{Discrete Contin. Dyn. Syst.} {\bf 9} (2003), no.~3, 677--704}
       {}

\bibitem{imG.vbL1958}
\mcitej{I.~M.~Gelfand, V.~B.~Lidskii}
       {On the structure of the regions of stability of linear canonical systems of
        differential equations with periodic coefficients}
       {in: ``Twelve Papers on Function Theory, Probability and Differential Equations'',
        American Mathematical Society Translations, Ser.~2, Vol.~8, pp.~143--181,
        American Mathematical Society, Providence, RI, 1958}
       {}

\bibitem{pH.sJ.bK2018}
\mcitej{P.~Howard, S.~Jung, B.~Kwon}
       {The Maslov index and spectral counts for linear Hamiltonian systems on $[0,1]$}
       {\jour{J. Dynam. Differential Equations} {\bf 30} (2018), no.~4, 1703--1729}
       {}

\bibitem{pH.yL.aS2017}
\mcitej{P.~Howard, Y.~Latushkin, A.~Sukhtayev}
       {The Maslov index for Lagrangian pairs on ${\mathbb R}^{2n}$}
       {\jour{J. Math. Anal. Appl.} {\bf 451} (2017), no.~2, 794--821}
       {}

\bibitem{pH.yL.aS2018}
\mcitej{P.~Howard, Y.~Latushkin, A.~Sukhtayev}
       {The Maslov and Morse indices for system Schr\"odinger operators on $\mathbb{R}$}
       {\jour{Indiana Univ. Math. J.} {\bf 67} (2018), no.~5, 1765--1815}
       {}

\bibitem{rJ.sN.cN.rO2017}
\mcitej{R.~Johnson, S.~Novo, C.~N\'u\~nez, R.~Obaya}
       {Nonautonomous linear-quadratic dissipative control processes without uniform null
        controllability}
       {\jour{J. Dynam. Differential Equations} {\rm 29} (2017), no.~2, 355--383}
       {}

\bibitem{rJ.rO.sN.cN.rF2016}
\mciteb{R.~Johnson, R.~Obaya, S.~Novo, C.~N\'u\~nez, R.~Fabbri}
       {Nonautonomous Linear Hamiltonian Systems: Oscillation, Spectral Theory and Control}
       {Developments in Mathematics, Vol.~36, Springer, Cham, 2016}
       {}

\bibitem{wK1993}
\mcitej{W.~Kratz}
       {A limit theorem for monotone matrix functions}
       {\jour{Linear Algebra Appl.} {\bf 194} (1993), 205--222}
       {}

\bibitem{wK1995}
\mciteb{W.~Kratz}
       {Quadratic Functionals in Variational Analysis and Control Theory}
       {Mathematical Topics, Vol.~6, Akademie Verlag, Berlin, 1995}
       {}
\bibitem{wK2003} W. Kratz, Definiteness of quadratic functionals, Analysis (Munich) 23(2), 163--183 (2003).
\bibitem{wK.rSH2012}
\mcitej{W.~Kratz, R.~\v{S}imon~Hilscher}
       {Rayleigh principle for linear Hamiltonian systems without controllability}
       {\jour{ESAIM Control Optim. Calc. Var.} {\bf 18} (2012), no.~2, 501--519}
       {}

\bibitem{wK.rSH2013}
\mcitej{W.~Kratz, R.~\v{S}imon~Hilscher}
       {A generalized index theorem for monotone matrix-valued functions with
        applications to discrete oscillation theory}
       {\jour{SIAM J. Matrix Anal. Appl.} {\bf 34} (2013), no.~1, 228–-243}
       {}

\bibitem{vbL1955}
\mcitej{V.~B.~Lidskii}
       {Oscillation theorems for canonical systems of differential equations}
       {(Russian) \jour{Dokl. Akad. Nauk SSSR (N.S.)} {\rm 102} (1955), no.~5, 877--880.
        Translation in: \jour{NASA Technical Translation}, TT F-14, 696 (1973), 9 pp}
       {}

\bibitem{pS.rSH2017}
\mcitej{P.~\v{S}epitka, R.~\v{S}imon~Hilscher}
       {Comparative index and Sturmian theory for linear Hamiltonian systems}
       {\jour{J. Differential Equations} {\bf 262} (2017), no.~2, 914--944}
       {}

\bibitem{pS.rSH2018a}
\mcitej{P.~\v{S}epitka, R.~\v{S}imon~Hilscher}
       {Singular Sturmian separation theorems for nonoscillatory symplectic difference
        systems}
       {\jour{J. Difference Equ. Appl.} {\bf 24} (2018), no.~12, 1894--1934}
       {}

\bibitem{pS.rSH2019}
\mcitej{P.~\v{S}epitka, R.~\v{S}imon~Hilscher}
       {Singular Sturmian separation theorems on unbounded intervals for linear Hamiltonian
        systems}
       {\jour{J. Differential Equations} {\bf 266} (2019), no.~11, 7481--7524}
       {}

\bibitem{pS.rSH2020}
\mcitej{P.~\v{S}epitka, R.~\v{S}imon~Hilscher}
       {Singular Sturmian comparison theorems for linear Hamiltonian systems}
       {\jour{J. Differential Equations} {\bf 269} (2020), no.~4, 2920--2955}
       {}

\bibitem{pS.rSH2021}
\mcitej{P.~\v{S}epitka, R.~\v{S}imon~Hilscher}
       {Distribution and number of focal points for linear Hamiltonian systems}
       {\jour{Linear Algebra Appl.} {\bf 611} (2021), 26--45}
       {}

\bibitem{pS.rSH2020?e}
\mcitej{P.~\v{S}epitka, R.~\v{S}imon~Hilscher}
       {Lidskii angles and Sturmian theory for linear Hamiltonian systems on compact
        interval}
       {to appear}
       {}

\bibitem{pS.rSH2020?g}
\mcitej{P.~\v{S}epitka, R.~\v{S}imon~Hilscher}
       {Comparative index and Lidskii angles for symplectic matrices}
       {submitted (2020)}
       {\jour{Linear Algebra Appl.} {\bf 624} (2021), 174--197}

\bibitem{mW2007}
\mcitej{M.~Wahrheit}
       {Eigenvalue problems and oscillation of linear Hamiltonian systems}
       {\jour{Int. J. Difference Equ.} {\bf 2} (2007), no.~2, 221--244}
       {}

\bibitem{vaY1961}
\mcitej{V.~A.~Yakubovich}
       {Arguments on the group of symplectic matrices}
       {(Russian) \jour{Mat. Sb. (N.S.)} {\bf 55 (97)} (1961), no.~3, 255--280}
       {}

\bibitem{vaY1964}
\mcitej{V.~A.~Yakubovich}
       {Oscillatory properties of solutions of canonical equations}
       {in: ``Fifteen Papers on Differential Equations'', American Mathematical Society
        Translations, Ser.~2, Vol.~42, pp.~247--288, American Mathematical Society,
        Providence, RI, 1964}
       {}

\bibitem{yZ.lW.cZ2018}
\mcitej{Y.~Zhou, L.~Wu, C.~Zhou}
       {H\"ormander index in finite-dimensional case}
       {\jour{Front. Math. China} {\bf 13} (2018), no.~3, 725--761}
       {}

\end{thebibliography}
\end{document}